\documentclass[reqno,makeidx]{amsart}

\usepackage{graphicx}
\usepackage{latexsym}
\usepackage{amstext}
\usepackage {amsmath}
\usepackage {amsfonts}
\usepackage {amssymb}
\usepackage {amsthm}
\usepackage {bbm}
\usepackage{enumerate}
\usepackage{array}
\usepackage{comment}
\usepackage{xspace}

\DeclareMathOperator{\dist}{dist}

\DeclareMathOperator{\spt}{supp}

\DeclareMathOperator{\dive}{div}
\def\loc{{\mathrm{loc}}}

\newcommand{\eps}{\varepsilon}

\newcommand{\Chi}{\mathcal{X}}
\newcommand{\R}{\ensuremath{\mathbb{R}}}
\newcommand{\Rn}{\ensuremath{{\mathbb{R}^n}}}
\newcommand{\nablatx}{\ensuremath{{\nabla^\prime}}}
\newcommand{\N}{\ensuremath{\mathbb{N}}}
\newcommand{\LL}{\ensuremath{\mathcal{L}}}

\newcommand{\Ha}{\ensuremath{\mathcal{H}}}

\newcommand{\proj}{\Pi}

\def\R{\mathbb R}

\def\W{\mathcal{W}}
\def\A{\mathcal{S}}
\def\AK{\mathcal{S}^{0}}
\def\Laction{\Lambda_1}
\def\Linit{\Lambda_2}
\def\Lsquares{\Lambda_3}
\def\Lenergy{\Lambda_4}

\def\B{\mathcal{B}}
\def\Per{\mathcal{P}}
\def\Prob{\mathsf{Prob}}

\theoremstyle{plain}
\numberwithin{equation}{section}
\newtheorem{lemma}{Lemma}[section]
\newtheorem{theorem}[lemma]{Theorem}
\newtheorem{proposition}[lemma]{Proposition}
\newtheorem{definition}[lemma]{Definition}
\newtheorem{assumption}[lemma]{Assumption}

\theoremstyle{definition}
\newtheorem{remark}[lemma]{Remark}
\begin{document}
\title{The Allen--Cahn Action functional in higher dimensions}
\author{Luca Mugnai}
\address{Luca Mugnai, Max Planck Institute for Mathematics in the
  Sciences, Inselstr. 22, D-04103 Leipzig}

\author{Matthias R{\"o}ger}
\address{Matthias R\"{o}ger, Max Planck Institute for Mathematics in the
  Sciences, Inselstr. 22, D-04103 Leipzig}

\email{mugnai@mis.mpg.de, roeger@mis.mpg.de}

\subjclass[2000]{Primary 49J45; Secondary 35R60, 60F10, 53C44}

\keywords{Allen-Cahn equation, stochastic partial differential equations, large
deviation theory, sharp interface limits, motion by mean curvature}

\date{\today}

\begin{abstract}
The Allen--Cahn action functional is related to the probability of rare
events in the stochastically perturbed Allen--Cahn equation. Formal
calculations suggest a \emph{reduced action functional} in the sharp
interface limit. We prove the
corresponding lower bound in two and three space dimensions. 
One difficulty is that diffuse interfaces may
collapse in the limit. We therefore consider the limit of diffuse
surface area measures and introduce a generalized velocity and
generalized reduced action functional in a class of evolving
measures.
\end{abstract}

\maketitle
\section{Introduction}
\label{sec:intro}
In this paper we study the (renormalized) \emph{Allen--Cahn
  action functional}
\begin{gather}
 \A_\eps(u)\,:=\, \int_0^T\int_\Omega \Big(\sqrt{\eps}\partial_t u
  +\frac{1}{\sqrt{\eps}}\big(-\eps\Delta u 
  +\frac{1}{\eps}W^\prime(u)\big)\Big)^2\,dx\,dt. \label{def:action}
\end{gather}
This functional arises in the analysis of the stochastically
perturbed Allen--Cahn equation
\cite{ALan03,KKLa07,FHSv04,Shar00,ERVa04,FJon82,Feng06}  and is related to the
probability of rare 
events such as switching between
deterministically stable states.

Compared to the purely deterministic setting, stochastic perturbations add
new features to the theory of phase separations, and the analysis of 
action functionals has drawn attention 
\cite{ERVa04,FHSv04,KORV,KRT,RTon07}. Kohn \emph{et alii} \cite{KORV}
considered the \emph{sharp-interface limit}
$\eps\to 0$ of $\A_\eps$ and identified a \emph{reduced action
  functional} that is more easily
accessible for a qualitative analysis. The sharp
interface limit reveals a connection between minimizers of $\A_\eps$ and
mean curvature flow. 

The reduced action functional in  \cite{KORV} is defined for
phase indicator functions $u:(0,T)\times\Omega\to\{-1,1\}$ with the
additional properties that the measure of the phase $\{u(t,\cdot)\,=\,1\}$ is continuous and
 the common boundary of the two phases $\{u=1\}$
and $\{u=-1\}$ is, apart from a countable set of singular times, given
as union of smoothly evolving 
hypersurfaces $\Sigma:= \cup_{t\in (0,T)}\{t\}\times\Sigma_t$. The
reduced action functional is then defined as
\begin{align}
  \AK(u)\,&:=\, c_0\int_0^T\int_{\Sigma_t} \big|v(t,x) -
  H(t,x)\big|^2 \,d\Ha^{n-1}(x)dt\, + 4\AK_{nuc}(u),
  \label{def:red-action} \\
  \AK_{nuc}(u)\,&:=\, 2c_0 \sum_{i}\Ha^{n-1}(\Sigma_i), \label{def:AK-nuc}
\end{align}
where $\Sigma_i$ denotes the $i^{th}$ component of $\Sigma$ at
the time of creation, where $v$ denotes the normal velocity of
the evolution $(\Sigma_t)_{t\in (0,T)}$, where $H(t,\cdot)$ denotes the
mean curvature vector of $\Sigma_t$, and where the constant $c_0$ is determined
by $W$, 
\begin{gather}\label{eq:def-c0}
  c_0 \,:=\, \, \int_{-1}^1 \sqrt{2W(s)}\,ds.
\end{gather}
(See Section \ref{sec:discussion} for a more rigourous definition of $\AK$).

Several arguments suggest that $\AK$ describes the
Gamma-limit of $\A_\eps$: 
\begin{itemize}
\item
The \emph{upper bound} necessary for the Gamma-convergence was
formally proved \cite{KORV} by the construction of good `recovery
sequences'. 
\item
The \emph{lower bound} was proved in \cite{KORV} for sequences
$(u_\eps)_{\eps>0}$ such that the 
associated `energy-measures' have \emph{equipartitioned energy} and
\emph{single multiplicity} as $\eps\to 0$.
\item
In one space-dimension Reznikoff and Tonegawa \cite{RTon07} proved that
$\A_\eps$ Gamma-converges to an appropriate relaxation 
of the one-dimensional version of $\AK$.
\end{itemize}
The approach used in \cite{KORV} is based on
the evolution of the phases and is sensible to
cancellations of phase boundaries in the sharp interface
limit. Therefore in \cite{KORV} a sharp lower bound is  achieved 
only under a single-multiplicity assumption for
the limit of the diffuse interfaces. As a consequence,
it could not be excluded that creating multiple interfaces
reduces the action.

In the present paper we prove a sharp lower-bound of the functional
$\A_\eps$ in space dimensions $n=2,3$ without any additional
restrictions on the approximate sequences.

To circumvent
problems with cancellations of interfaces we analyze the
evolution of the (diffuse) \emph{surface-area measures}, which makes
information available that is lost in the limit of phase
fields.
With this aim we generalize the functional $\AK$
to a suitable class of \emph{evolving energy measures} and introduce a
generalized formulation of velocity, similar to Brakke's
generalization of Mean Curvature Flow \cite{Brak78}.

Let us informally describe our approach and main results. Comparing the
two functionals $\A_\eps$ and $\AK$ the first and second term of the sum
in the integrand \eqref{def:action} describe a `diffuse velocity' and
`diffuse mean curvature' respectively. We will make this statement
precise in \eqref{eq:conv-pair} and \eqref{eq:conv-H-eps}. The mean
curvature is given by the first variation of the area functional, and a
lower estimate for the square of the diffuse mean curvature is available
in a time-independent situation \cite{RS}. The velocity of the
evolution of the phase boundaries is determined by the time-derivative
of the surface-area measures and the nucleation term in the functional
$\AK$ in fact describes a singular part of this time derivative.

Our first main result  is a compactness
result: the diffuse surface-area measures converge to an evolution of
measures with a square integrable generalized mean curvature and a
square integrable generalized velocity . In the class of such evolutions
of measures we provide a 
generalized formulation of the reduced action functional. We
prove a lower estimate that counts the propagation cost
with the multiplicity of the interface. This shows
that it is more expensive to move phase boundaries with higher
multiplicity. Finally we prove two statements on the Gamma-convergence
(with respect to $L^1(\Omega_T)$) of the action functional. The first result is
for evolutions in the domain of $\AK$ that have nucleations only at
the initial time. This is in particular desirable since minimizers of $\AK$
are supposed to be in this class. The second result proves the Gamma
convergence in $L^1(\Omega_T)$ under an assumption on the 
structure of the set of measures arising as sharp interface limits of
sequences with uniformly bounded action.

We give a precise statement of our main results in Section
\ref{sec:results}. In the remainder of this introduction we describe
some background and motivation.

\subsection{Deterministic phase field models and sharp interface
limits} 
Most \emph{diffuse interface models} are based on the 
\emph{Van~der~Waals--Cahn--Hilliard} energy
\begin{gather}
  E_{\varepsilon}(u)\,:=\,\int_{\Omega}\Big(\frac{\varepsilon}{2}|\nabla u|^2 
  +\frac{1}{\varepsilon}W(u)\Big)\, dx. \label{eq:def-E-intro}
\end{gather}
The energy $E_\eps$
favors a decomposition of $\Omega$ into two regions (phases) where
$u\approx -1$ and $ u\approx 1$, separated by a transition layer (diffuse
interface) with a thickness of order $\eps$. Modica and Mortola
\cite{MM,Mo} proved that $E_\eps$  
Gamma-converges (with respect to $L^1$-convergence) to a constant
multiple of the perimeter functional $\Per$, restricted to phase indicator
functions,
\begin{gather*}
  E_\eps\,\to\, c_0\Per,\qquad
  \Per(u)\,:=\, 
  \begin{cases}
    \frac{1}{2}\int_\Omega \,d|\nabla u| &\text{ if } u\in BV(\Omega,\{-1,1\}),\\
    \infty &\text{ otherwise.}
  \end{cases}
\end{gather*}
$\Per$ measures the surface-area of the phase boundary
$\partial^*\{u=1\}\cap\Omega$. In this sense $E_\eps$ describes a diffuse
approximation of the surface-area functional.

Various tighter connections between the functionals $E_\eps$ and $\Per$
have been proved. We mention here just two that are important for our
analysis. The (accelerated) $L^2$-gradient flow of $E_\eps$ is
given by the \emph{Allen--Cahn equation}
\begin{gather}
  \eps \partial_t u \,=\, \eps\Delta u -\frac{1}{\eps}W'(u)
  \label{eq:AC}
\end{gather}
for phase fields in the time-space cylinder $(0,T)\times\Omega$. It is
proved in different formulations \cite{dMSc90,ESSo92,Ilma93} that 
\eqref{eq:AC} converges to the \emph{Mean Curvature Flow}
\begin{gather}
  H(t,\cdot)\,=\, v(t,\cdot) \label{eq:MCF}
\end{gather}
for the evolution of phase boundaries.

Another connection between the first variations of $E_\eps$ and $\Per$ is
expressed in a (modified) conjecture of De Giorgi \cite{DG}: Considering
\begin{gather}
  \W_\eps(u)\,:=\,\int_\Omega \frac{1}{\eps}\Big(-
  \eps\Delta u +\frac{1}{\eps}W'(u)\Big)^2\,dx \label{eq:def-W-intro}
\end{gather}
the sum $E_\eps+\W_\eps$ Gamma-converges up to the constant factor $c_0$ 
to the sum of the Perimeter functional and the \emph{Willmore functional} $\W$,
\begin{gather}
  E_\eps + \W_\eps\,\to\, c_0\Per + c_0\W,\qquad
  \W(u)\,=\, \int_{\Gamma} H^2\,d\Ha^{n-1}, \label{eq:def-willmore}
\end{gather}
where $\Gamma$ denotes the phase boundary $\partial^*\{u=1\}\cap\Omega$.
This statement was recently proved by R\"oger and Sch\"atzle \cite{RS}
in space dimensions $n=2,3$ and is one essential ingredient to obtain
the lower bound for the action functional.

\subsection{Stochastic interpretation of the action functional}
Phenomena such as the nucleation of a
new phase or the switching between two (local) energy minima require an
energy barrier crossing and are out of the scope of deterministic models
that are energy dissipative.
If thermal fluctuations are taken into account such an
energy barrier crossing becomes possible.
In \cite{KORV} `thermally activated switching' was considered
for the \emph{stochastically
perturbed Allen--Cahn equation}
\begin{gather}
  \eps \partial_t u \,=\, \eps\Delta u -\frac{1}{\eps}W'(u) +
  \sqrt{2\gamma} \eta_\lambda
  \label{eq:stoch-AC}
\end{gather}
Here $\gamma>0$ is a parameter that represents the temperature of the
system, $\eta$ is a time-space white noise, and $\eta_\lambda$ is a spatial
regularization with $\eta_\lambda\to \eta$ as $\lambda\to 0$. This
regularization is necessary for $n\geq 2$ since the white noise is too
singular to ensure well-posedness of \eqref{eq:stoch-AC} in higher
space-dimensions.

Large deviation theory and (extensions of) results by
Wentzell and Freidlin \cite{FWen98,WFre79} yield an estimate on the
probability distribution of solutions of stochastic ODEs and PDEs in the
small-noise limit. This estimate is expressed in terms of a
(deterministic) action functional. For instance, 
thermally activated switching within a time $T>0$ is described by the
set of paths 
\begin{gather}
  \B\,:=\, \Big\{u(0,\cdot)\,=\, -1,\quad
  \|u(t,\cdot)-1\|_{L^\infty(\Omega)}\,\leq\,\delta\text{ for some
  }t\leq T \Big\}, \label{eq:prob-switch}
\end{gather}
where $\delta>0$ is a fixed constant.
The probability of switching for solutions of \eqref{eq:stoch-AC} then
satisfies 
\begin{gather}
  \lim_{\gamma\to 0} \gamma \ln \Prob(\B)\,=\, -\inf_{u\in\B}
  \A_\eps^{(\lambda)}(u). \label{eq:est-action-lambda}
\end{gather}
Here $\A_\eps^{(\lambda)}$ is the action functional associated to
\eqref{eq:stoch-AC} and converges (formally) to the action functional
$\A_\eps$ as $\lambda\to 0$ \cite{KORV}.
Large deviation theory not only estimates the probability
of rare events but also identifies the `most-likely switching path'  as
the minimizer $u$ in \eqref{eq:est-action-lambda}.

We focus here on the sharp interface limit $\eps\to 0$ of the
action functional $\A_\eps$. The small parameter $\eps>0$ corresponds to
a specific diffusive scaling of the time- and space domains. This choice
was identified \cite{ERVa04,KORV} as particularly 
interesting,  exhibiting a competition between \emph{nucleation versus
  propagation} to achieve the optimal switching. Depending on the value of
$|\Omega|^{1/d}/\sqrt{T}$ a cascade of more and more complex spatial patterns is
observed \cite{ERVa04,KORV,KRT}. The interest in the sharp interface 
limit is motivated by an interest in applications
where the switching time is small compared to the deterministic
time-scale, see for instance \cite{KRVa05}.

\subsection{Organization}
We fix some notation and assumptions in the next
section. In Section \ref{sec:L2-flows} we introduce the concept of
$L^2$-flows and generalized velocity. Our main results are stated in Section
\ref{sec:results} and proved in the Sections
\ref{sec:proofs-prelim}-\ref{sec:proofs-velo}. We discuss some
implications for the Gamma-convergence of the
action functional in Section \ref{sec:discussion}. Finally, in the Appendix we
 collect some definitions from Geometric Measure Theory.
\subsection*{Acknowledgment}
We wish to thank Maria Reznikoff, Yoshihiro Tonegawa, and Stephan Luckhaus for
several stimulating discussions. The first author thanks the Eindhoven
University of Technology for its hospitality during his stay in
summer 2006.

The first author was partially supported by the
\emph{Schwerpunktprogramm DFG SPP 1095 `Multiscale Problems'} and
\emph{DFG Forschergruppe 718}. 
\section{Notation and Assumptions}
\label{sec:nota}
Throughout the paper we will adopt the following notation: $\Omega$ is an
open bounded subset of $\R^n$ with Lipschitz boundary; $T>0$ is a real number and
$\Omega_T:=(0,T)\times\Omega$; $x\in\Omega$ and 
$t\in (0,T)$ denote the space- and time-variables respectively; $\nabla$
and $\Delta$ denote the spatial gradient and Laplacian and $\nablatx$
the full gradient in 
$\R\times\Rn$.

We choose $W$ to be the standard quartic double-well potential
\begin{gather*}
  W(r)\,=\,\frac{1}{4}(1-r^2)^2.
\end{gather*}

For a family of measures $(\mu^t)_{t\in (0,T)}$ we denote by $\LL^1\otimes
\mu^t$ the product measure defined by
\begin{gather*}
  \big(\LL^1\otimes \mu^t\big)(\eta)\,:=\, \int_0^T \mu^t(\eta(t,\cdot))\,dt 
\end{gather*}
for any $\eta\in C^0_c(\Omega_T)$.

We next state our main assumptions.
\begin{assumption}\label{ass:main}
Let $n=2,3$ and let a sequence $(u_\eps)_{\eps>0}$ of smooth functions
be given that satisfies for all $\eps>0$ 
\begin{gather}
  \A_\eps(u_\eps) \,\leq\,
  \Laction, \tag{A1} \label{ass:bound}\\
  \int_{\Omega} \big(
  \frac{\eps}{2}|\nabla u_\eps|^2
  +\frac{1}{\eps}W(u_\eps)\big)(0,x)\,dx
  \,\leq\, \Linit, \tag {A2} \label{ass:init}
\end{gather}
where the constants $\Laction, \Linit$ are independent
of $\eps>0$. Moreover we prescribe that
\begin{gather}
  \nabla u_\eps\cdot\nu_\Omega\,=\,0\quad\text{ on
  }[0,T]\times\partial\Omega. \tag{A3} \label{eq:neumann}
\end{gather}
\end{assumption}

\begin{remark}\label{rem:ass-KRT}
It follows from \eqref{eq:neumann} that for any $0\leq t_0\leq T$
\begin{align*}
  &\int_0^{t_0}\int_\Omega \Big(\sqrt{\eps}\partial_t u_\eps
  +\frac{1}{\sqrt{\eps}}\big(-\eps\Delta u_\eps 
  +\frac{1}{\eps}W^\prime(u_\eps)\big)\Big)^2\,dxdt\\
  =\,&
  \int_0^{t_0}\int_{\Omega} \eps (\partial_t u_\eps)^2 +\frac{1}{\eps} \big(-\eps\Delta u_\eps +
  \frac{1}{\eps}W^\prime(u_\eps)\big)^2\,dxdt\\
  &+ 2\int_{\Omega} \big(
  \frac{\eps}{2}|\nabla u_\eps|^2
  +\frac{1}{\eps}W(u_\eps)\big)(t_0,x)\,dx - 2\int_{\Omega} \big(
  \frac{\eps}{2}|\nabla u_\eps|^2
  +\frac{1}{\eps}W(u_\eps)\big)(0,x)\,dx.
\end{align*}
By the uniform bounds \eqref{ass:bound}, \eqref{ass:init} this implies that
\begin{gather}
  \int_{\Omega_T} \eps (\partial_t u_\eps)^2 +\frac{1}{\eps} \big(-\eps\Delta u_\eps +
  \frac{1}{\eps}W^\prime(u_\eps)\big)^2\,dxdt\,\leq\, \Lsquares,
  \label{ass:bound-squares}\\ 
  \max_{0\leq t\leq T}\int_{\Omega} \big(
  \frac{\eps}{2}|\nabla u_\eps|^2
  +\frac{1}{\eps}W(u_\eps)\big)(t,x)\,dx
  \,\leq\, \Lenergy,
  \label{ass:bound-mu}
\end{gather}
where
\begin{gather*}
  \Lsquares\,:=\, \Laction + 2\Linit ,\qquad
  \Lenergy\,:=\,\frac{1}{2}\Laction + \Linit.
\end{gather*}
\end{remark}
\begin{remark}
Our arguments would also work for any boundary conditions for which
$\partial_t u \nabla u\cdot\nu_\Omega$ vanishes on $\partial\Omega$,
in particular for time-independent Dirichlet conditions or periodic
boundary conditions. 
\end{remark}

We set
\begin{gather}
  w_\eps\,:=\, -\eps\Delta u_\eps +
  \frac{1}{\eps}W^\prime(u_\eps) \label{eq:def-w}
\end{gather}
and define for $\eps>0$, $t\in (0,T)$ a Radon measure $\mu_\eps^t$ on
$\overline{\Omega}$, 
\begin{align}
  \mu_\eps^t \,&:=\, \Big(\frac{\eps}{2}|\nabla u_\eps|^2(t,\cdot)
  +\frac{1}{\eps}W(u_\eps(t,\cdot))\Big)\LL^n, \label{eq:def-mu-eps,t}
\end{align}
and for $\eps>0$ measures $\mu_\eps,\alpha_\eps$ on $\overline{\Omega_T}$,
\begin{align}
  \mu_\eps \,&:=\, \Big(\frac{\eps}{2}|\nabla u_\eps|^2
  +\frac{1}{\eps}W(u_\eps)\Big)\LL^{n+1}, \label{eq:def-mu-eps}\\ 
  \alpha_\eps\,&:=\, \big(\eps^{1/2}\partial_t
  u_\eps+\eps^{-1/2}w_\eps\big)^2 \LL^{n+1}. \label{eq:def-alpha-eps}
\end{align}

Eventually restricting ourselves to a subsequence $\eps\to 0$ we may
assume that
\begin{alignat}{2}
  \mu_\eps\,&\to\, \mu \quad&&\text{ as Radon-measures on
  }{\overline{\Omega_T}}, 
  \label{eq:conv-mu}\\ 
  \alpha_\eps\,&\to\, \alpha &&\text{ as Radon-measures on
  }{\overline{\Omega_T}}, 
  \label{eq:conv-alpha}
\end{alignat}
for two Radon measures $\mu, \alpha$ on $\overline{\Omega_T}$,
and that
\begin{gather}
  \alpha(\overline{\Omega_T})\,=\, \liminf_{\eps\to
  0}\alpha_\eps(\Omega_T). \label{eq:ass-liminf} 
\end{gather}
\section{$L^2$-flows}
\label{sec:L2-flows}
We will show that the uniform bound on the action implies the existence 
of a square-integrable weak mean curvature and the existence of a
square-integrable \emph{generalized velocity}. 
The formulation of weak mean curvature is standard in Geometric Measure
Theory \cite{Alla72,Si}. Our definition of $L^2$-flow and generalized
velocity is similar to Brakke's formulation of mean curvature flow
\cite{Brak78}. 
\begin{definition}
\label{def:gen-velo}
Let $(\mu^t)_{t\in (0,T)}$ be any family of integer rectifiable Radon measures
such that $\mu:=\LL^1\otimes\mu^t$ defines a Radon measure on $\Omega_T$ and
such that $\mu^t$ has a weak mean curvature $H(t,\cdot)\in L^2(\mu^t)$
for almost all $t\in (0,T)$.

If there exists a positive
constant $C$ and 
a vector field $v\in L^2(\mu,\Rn)$ such that
\begin{gather}
  v(t,x)\perp T_x\mu^t\quad\text{ for }\mu\text{-almost all
  }(t,x)\in\Omega_T, \label{eq:def_vel_perp}\\
  \label{eq:def_vel_L2_flow}
  \Big\vert \int_0^T\int_\Omega \big(\partial_t\eta+
  \nabla\eta\cdot v\big)\,d\mu^tdt \Big\vert \,\leq\,
  C\|\eta\|_{C^0(\Omega_T)}
\end{gather}
for all $\eta\in C^1_c((0,T)\times\overline{\Omega})$, then we call  the
evolution 
$(\mu^t)_{t\in (0,T)}$ an \emph{$L^2$-flow}. A function $v\in
L^2(\mu,\Rn)$ satisfying \eqref{eq:def_vel_perp},
\eqref{eq:def_vel_L2_flow} is called a 
\emph{generalized velocity vector}.
\end{definition}
This definition is based on the observation that for a smooth evolution
$(M_t)_{t\in (0,T)}$ with mean curvature $H(t,\cdot)$ and normal
velocity vector $V(t,\cdot)$
\begin{align*}
  &\frac{d}{dt}\int_{M_t}\eta(t,x)\,d\Ha^{n-1}(x) -
  \int_{M_t}\partial_t\eta(t,x)\,d\Ha^{n-1}(x)  -
  \int_{M_t}\nabla\eta(t,x)\cdot V(t,x)\,d\Ha^{n-1}(x) \\
  =\,&  \int_{M_t}H(t,x)\cdot V(t,x)\eta(t,x)\,d\Ha^{n-1}(x).
\end{align*}
Integrating this equality in time implies \eqref{eq:def_vel_L2_flow}  for
any evolution with 
square-integrable velocity and mean curvature. 
\begin{remark}
\label{rem:singular-times}
Choosing $\eta(t,x)=\zeta(t)\psi(x)$ with $\zeta\in C^1_c(0,T)$,
$\psi\in C^1(\overline{\Omega})$, we deduce from \eqref{eq:def_vel_L2_flow}
that $t\mapsto \mu^t(\psi)$ belongs to $BV(0,T)$. Choosing a countable
dense subset $(\psi_i)_{i\in\N}\subset C^0(\overline{\Omega})$ this implies
that there exists a countable set $S\subset (0,T)$ of \emph{singular
times} such that any good representative of $t\mapsto\mu^t(\psi)$ is
continuous in $(0,T)\setminus S$ for all $\psi\in C^1(\overline{\Omega})$.
\end{remark}  

Any
generalized velocity is (in a set of good points) uniquely determined by
the evolution $(\mu_t)_{t\in (0,T)}$.  
\begin{proposition}
\label{prop:uni_gen_vel_perp} 
Let $(\mu^t)_{t\in(0,T)}$ be an $L^2$-flow and set
$\mu:=\LL^1\otimes\mu^t$. Let $v\in L^2(\mu)$ be a generalized velocity
field in the sense of Definition \ref{def:gen-velo}. 
Then 
\begin{equation}\label{eq:gen_v_perp}
  \begin{pmatrix} 1 \\ v(t_0,x_0)\end{pmatrix} \in T_{(t_0,x_0)}\mu
\end{equation}
holds in $\mu$-almost all points $(t_0,x_0)\in\Omega_T$ where the
tangential plane of $\mu$ exists. The evolution
$(\mu^t)_{t\in(0,T)}$ uniquely determines $v$ in all points
$(t_0,x_0)\in \Omega_T$ where both tangential planes
$T_{(t_0,x_0)}\mu$ and $T_{x_0}\mu^{t_0}$ exist.
\end{proposition}
We postpone the proof to Section
\ref{sec:proofs-velo}.

In the set of points where a tangential plane of $\mu$ exists, the
 generalized velocity field $v$ coincides with
the normal velocity introduced in \cite{BMug07}. 

We turn now to the statement of a lower bound for sequences
$(u_\eps)_{\eps>0}$ satisfying Assumption \ref{ass:main}. As $\eps\to 0$
we will obtain a phase indicator function $u$ as the limit of the sequence
$(u_\eps)_{\eps>0}$ and an $L^2$-flow $(\mu^t)_{t\in(0,T)}$ as the limit of
the measures $(\mu_\eps)_{\eps>0}$. We will show
that in $\Ha^{n}$-almost all points of the phase boundary
$\partial^*\{u=1\}\cap\Omega$ a tangential plane of $\mu$ exists. This
implies the existence of a unique normal velocity field of the phase
boundary. 
\section{Lower bound for the action functional}\label{sec:results}
In several steps we state a lower bound for the functionals
$\A_\eps$. 
We postpone all proofs to Sections
\ref{sec:proofs-prelim}-\ref{sec:proofs-velo}.  

\subsection{Lower estimate for the mean curvature}
\label{subsec:lsc-mc}
We start with an application of the well-known results of Modica and Mortola 
\cite{MM,Mo}. 
\begin{proposition}\label{prop:MM}
There exists $u\in BV(\Omega_T,\{-1,1\})\cap L^\infty(0,T;BV(\Omega))$
such that for a subsequence $\eps\to 0$
\begin{alignat}{2}
  u_\eps\,&\to\, u&&\text{ in }L^1(\Omega_T),\label{eq:conv-MM}\\
  u_\eps(t,\cdot)\,&\to\, u(t,\cdot)\qquad&&\text{ in }L^1(\Omega)\text{
  for almost all }t\in (0,T). \label{eq:conv-MM-t}
\end{alignat}
Moreover
\begin{align}
  \frac{c_0}{2} \int_{\Omega_T}\,d\vert\nablatx u\vert \,&\leq\,\Lsquares + T\Lenergy, 
  \qquad
  \frac{c_0}{2}\int_\Omega \,d\vert\nabla u(t,\cdot)\vert\,\leq\,\Lenergy\label{eq:est-MM}
\end{align}
holds, where $c_0$ was defined in \eqref{eq:def-c0}.
\end{proposition}

The next proposition basically repeats the arguments in \cite[Theorem
1.1]{KRT}. 
\begin{proposition}\label{prop:KRT}
There exists a countable set $S\subset (0,T)$, a subsequence $\eps\to 0$
and Radon measures $\mu^t, t\in [0,T]\setminus S$, 
such that for all $t\in [0,T]\setminus S$
\begin{gather}
  \mu_\eps^t\,\to\, \mu^t\text{ as Radon measures on }\overline{\Omega},
  \label{eq:conv-mu-t}
\end{gather}
such that
\begin{gather}
  \mu\,=\, \LL^1\otimes \mu^t, \label{eq:disint-mu}
\end{gather}
and such that for all $\psi\in
C^1(\overline{\Omega})$ the function 
\begin{gather}
  t\mapsto \mu^t(\psi)\quad\text{ is of bounded variation in }(0,T)
  \label{eq:mu-bv} 
\end{gather}
and has no jumps in $(0,T)\setminus S$. 
\end{proposition}
Exploiting the lower bound \cite{RS} for the diffuse approximation of the Willmore
functional \eqref{eq:def-W-intro} we obtain that the measures $\mu^t$ are up to a
constant integer-rectifiable with a weak mean curvature satisfying an
appropriate lower estimate.
\begin{theorem}\label{the:lsc-H}
For almost all $t\in (0,T)$ 
\begin{align*}
  &\frac{1}{c_0}\mu^t \text{ is an integral }(n-1)\text{-varifold},\\
  &\mu^t\text{ has weak mean curvature }{H}(t,\cdot)\in L^2(\mu^t),
\end{align*}
and the estimate
\begin{gather}
   \int_{\Omega_T} |{H}|^2\,d\mu \, \leq\,
  \liminf_{\eps\to 0} \int_{\Omega_T}\frac{1}{\eps}w_\eps^2\,dxdt
  \label{eq:lsc-H}  
\end{gather}
holds.
\end{theorem}
\subsection{Lower estimate for the generalized velocity} 
\begin{theorem}\label{the:lsc-v}
Let $(\mu^t)_{t\in (0,T)}$ be the limit measures obtained in Proposition
\ref{prop:KRT}. Then there exists a generalized velocity $v\in
L^2(\mu,\Rn)$ of $(\mu^t)_{t\in (0,T)}$. Moreover the estimate
\begin{gather}
  \int_{\Omega_T} |v|^2\,d\mu\,\leq\, \liminf_{\eps\to 0} \int_{\Omega_T}
  \eps (\partial_t 
  u_\eps)^2\,dxdt \label{eq:v-lsc}
\end{gather}
is satisfied. In particular, $(\frac{1}{c_0}\mu^t)_{t\in (0,T)}$ is an $L^2$-flow.
\end{theorem}
We obtain $v$ as a limit of suitably defined approximate velocities, see
Lemma \ref{lem:velo}. On the phase boundary $v$
coincides with the 
(standard) distributional velocity of the bulk-phase $\{u(t,\cdot)=1\}$.
However, our definition extends the
velocity also to `hidden boundaries', which seems necessary in order to
prove the Gamma-convergence of the action functional; see the discussion
in Section \ref{sec:discussion}. 
\begin{proposition}
\label{prop:velo}
Define the \emph{generalized normal velocity} $V$ in direction of the
inner normal of $\{u=1\}$ by 
\begin{gather*}
  V(t,x)\,:=\, v(t,x)\cdot \frac{\nabla u}{|\nabla u|}(t,x),\quad
  \text{ for }(t,x)\in \partial^*\{u=1\}.
\end{gather*}
Then $V\in L^1(|\nabla u|)$
holds and $V|_{\partial^*\{u=1\}}$ is the unique
vector field that satisfies for all $\eta\in C^1_c(\Omega_T)$
\begin{equation}
  \int_0^T\int_{\Omega}V(t,x)\eta(t,x)\,d|\nabla u(t,\cdot)|(x) dt\,=\,
  -\,\int_{\Omega_T}u\partial_t\eta \,dxdt.
  \label{eq:def-v-diff}
\end{equation}
\end{proposition}
\subsection{Lower estimate of the action functional}
\label{subsec:lsc-action}
As our main result we obtain the following lower estimate for $\A_\eps$.
\begin{theorem}\label{the:main}
Let Assumption \ref{ass:main} hold, and let $\mu$,
$(\mu^t)_{t\in[0,T]}$, and  
$S$ be  the measures and the countable set of singular times that we
obtained in Proposition \ref{prop:KRT}.  
Define the \emph{ nucleation cost} $\A_{nuc}(\mu)$ by
\begin{equation}
\begin{split}
  \A_{nuc}(\mu)\,:=\,& \sum_{t_0\in S}
  \sup_{\psi}\Big(\lim_{t\downarrow t_0} \mu^t(\psi)-\lim_{t\uparrow
  t_0}\mu^t(\psi)\Big)
  \\
  +\,& \sup_\psi \big(\lim_{t\downarrow
  0}\mu^t(\psi)-\mu^0(\psi)\big) + \sup_\psi \big(\mu^T(\psi)
  -\lim_{t\uparrow T}\mu^t(\psi)\big),
  \label{eq:def-S-nuc} 
\end{split}
\end{equation}
where the $\sup$ is taken over all $\psi\in C^1(\overline{\Omega})$ with
$0\leq\psi\leq 1$.
Then
\begin{gather}
  \liminf_{\eps\to 0} \A_\eps(u_\eps)\,\geq\, \int_{\Omega_T}
  |v-H|^2\,d\mu + 4\A_{nuc}(\mu).\label{eq:lsc-S}
\end{gather}
\end{theorem}
In the previous definition of nucleation cost we have tacitly chosen
good representatives of $\mu^t(\psi)$ (see \cite{AFPa00}). With this
choice the jump parts in \eqref{eq:def-S-nuc} are well-defined.

Eventually let us remark that, in view of Theorem \ref{the:lsc-H}, we can conclude that 
$\mathcal S_{nuc}$ does indeed measure only $(n-1)$-dimensional jumps. 

Theorem \ref{the:main} improves \cite{KORV} in the higher-multiplicity
case. We will discuss our main results in Section
\ref{sec:discussion}.  
\subsection{Convergence of the Allen--Cahn equation to Mean curvature
  flow} 
\label{subsec:mcf}
Let $n=2,3$ and consider solutions $(u_\eps)_{\eps>0}$ of the
Allen--Cahn equation 
\eqref{eq:AC} satisfying \eqref{ass:init} and \eqref{eq:neumann}. Then
$S_\eps(u_\eps)=0$ and the results of Sections
\ref{subsec:lsc-mc}-\ref{subsec:lsc-action} apply: 
There exists a subsequence $\eps\to 0$ such that the phase functions
$u_\eps$ converge to a phase indicator function $u$, such that the energy
measures $\mu_\eps^t$ converge an $L^2$-flow $(\mu^t)_{t\in (0,T)}$, and
such that $\mu$-almost everywhere
\begin{gather}
  H \,=\, v \label{eq:lim-AC}
\end{gather}
holds, where $H(t,\cdot)$ denotes the weak mean curvature of $\mu^t$ and where
$v$ denotes the generalized velocity of $(\mu^t)_{t\in (0,T)}$ in the
sense of Definition \ref{def:gen-velo}. Moreover $\A_{nuc}(\mu)=0$, which
shows that for any nonnegative $\psi\in C^1(\overline{\Omega})$ the function
$t\mapsto \mu^t(\psi)$ cannot jump upwards. From
\eqref{eq:AC} and \eqref{eq:KRT-1} below one obtains that for any $\psi\in
C^1(\overline{\Omega})$ and all $\zeta\in C^1_c(0,T)$
\begin{gather}
  - \int_0^T \partial_t \zeta \, \mu^t_\eps(\psi)\,dt \,=\,
  -\int_{\Omega_T} \zeta(t)\Big(\frac{1}{\eps} \psi(x)w_\eps^2(t,x) +
  \nabla\psi(x)\cdot\nabla u_\eps w_\eps(t,x)\Big)\,dx
  dt. \label{eq:brakke-eps}
\end{gather}
We will show that suitably defined `diffuse
mean curvatures' converge as $\eps\to 0$, see
\eqref{eq:conv-H-eps}. Using this result we can pass to the limit in
\eqref{eq:brakke-eps} and we obtain for any nonnegative functions $\psi\in
C^1(\overline{\Omega})$, $\zeta\in C^1_c(0,T)$ that
\begin{gather*}
  - \int_0^T \partial_t \zeta \mu^t(\psi)\,dt \,\leq\,
  -\int_0^T \zeta(t)\int_{\Omega} \Big( H^2(t,x) + \nabla\psi(x)\cdot
  H(t,x)\Big)\,d\mu^t(x) dt, 
\end{gather*}
which is an time-integrated version of Brakke's inequality.
\section{Proofs of Propositions \ref{prop:MM}, \ref{prop:KRT} and Theorem
\ref{the:lsc-H}}
\label{sec:proofs-prelim}
\begin{proof}[Proof of Proposition \ref{prop:MM}]
By \eqref{ass:bound-squares}, \eqref{ass:bound-mu} we obtain that
\begin{gather*}
  \int_{\Omega_T} \Big(\frac{\eps}{2}|\nablatx u_\eps|^2
  +\frac{1}{\eps}W(u_\eps)\Big)\,dxdt\,\leq\, \Lsquares + T\Lenergy.
\end{gather*}
This implies by \cite{Mo} the existence of  a subsequence $\eps\to 0$
and of a function  $u\in BV(\Omega_T;\{-1,1\})$ such that
\begin{gather*}
  u_\eps \,\to\, u\quad\text{ in }L^1(\Omega_T)
\end{gather*}
and
\begin{gather*}
  \frac{c_0}{2}\int_{\Omega_T} \,d|\nablatx u|\,\leq\, \liminf_{\eps\to 0}
  \int_{\Omega_T}\big(\frac{\eps}{2}|\nablatx u_\eps|^2
  +\frac{1}{\eps}W(u_\eps)\big)\,dxdt  \,\leq\, \big(\Lsquares + T\Lenergy\big).
\end{gather*}
After possibly taking another subsequence, for almost all $t\in (0,T)$
\begin{gather}
  u_\eps(t,\cdot) \,\to\, u(t,\cdot)\quad\text{ in
  }L^1(\Omega) \label{eq:conv-ut} 
\end{gather}
holds.
Using \eqref{ass:bound-mu} and applying \cite{Mo} for a fixed $t\in
(0,T)$ with \eqref{eq:conv-ut} we get that
\begin{gather*}
  \frac{c_0}{2}\int_{\Omega} \,d|\nabla u|(t,\cdot)\,\leq\, \liminf_{\eps\to 0}
  \mu_\eps^t(\Omega)  \,\leq\,\Lenergy. 
\end{gather*}
\end{proof}
Before proving Proposition \ref{prop:KRT} we show that the
time-derivative of the energy-densities $\mu_\eps^t$ is controlled.
\begin{lemma}
\label{lem:bv-mu}
There exists $C=C(\Laction,\Lsquares,\Lenergy)$ such that for all
$\psi\in C^1(\overline{\Omega})$ 
\begin{gather}
  \int_0^T |\partial_t \mu_\eps^t(\psi)|\,dt\,\leq\,
  C\|\psi\|_{C^1(\overline{\Omega})}. \label{eq:bv-mu}
\end{gather}
\end{lemma}
\begin{proof}
Using \eqref{eq:neumann} we compute that 
\begin{align}
  2\partial_t\mu_\eps^t(\psi) \,
  =\, &\int_\Omega \big(\sqrt{\eps}\partial_t u_\eps
  +\frac{1}{\sqrt{\eps}}w_\eps\big)^2(t,x)\psi(x)\, dx - \int_{\Omega}
  \big(\eps(\partial_t u_\eps)^2 
  +\frac{1}{\eps}w_\eps^2\big)(t,x)\psi(x) \,dx\notag\\
  &  - 2\int_{\Omega} \eps\nabla\psi(x)\cdot \partial_t u_\eps(t,x)\nabla
  u_\eps(t,x)\, dx. \label{eq:KRT-1} 
\end{align}
By \eqref{ass:bound-squares}, \eqref{ass:bound-mu} we estimate
\begin{align}
  \Big|2\int_{\Omega_T} \eps\nabla\psi\cdot \partial_t u_\eps\nabla
  u_\eps\,dxdt\Big|\,\leq\,&
  \int_{\Omega_T} |\nabla\psi|\big(\eps (\partial_t u_\eps)^2 +\eps
  |\nabla u_\eps|^2\big)\,dxdt
  \notag\\
  \leq\,&
 (\Lsquares+T\Lenergy)\|\nabla\psi\|_{C^0(\overline{\Omega})} \label{eq:KRT-2}
\end{align}
and deduce from \eqref{ass:bound}, \eqref{ass:bound-squares},
\eqref{eq:KRT-1} that 
\begin{gather*}
  \int_0^T |\partial_t\mu_\eps^t(\psi)|\,dt\,\leq\,
  (\Laction+\Lsquares) \|\psi\|_{C^0(\overline{\Omega})} +
  C(\Lsquares,T\Lenergy)\|\nabla\psi\|_{C^0(\overline{\Omega})},
\end{gather*}
which proves \eqref{eq:bv-mu}.  
\end{proof}
\begin{proof}[Proof of Proposition \ref{prop:KRT}]
By \eqref{eq:conv-mu} $\mu_\eps\to\mu$ as Radon-measures on
$\overline{\Omega_T}$. 
Choose now a countable family $(\psi_i)_{i\in\N}\subset
C^1(\overline{\Omega})$ which is dense in $C^0(\overline{\Omega})$. By Lemma
\ref{lem:bv-mu} and a diagonal-sequence argument there exists a
subsequence $\eps\to 0$ and functions $m_i\in BV(0,T)$, $i\in\N$, such that
for all $i\in\N$
\begin{alignat}{2}
 \mu_\eps^t(\psi_i)\,&\to\, m_i(t)\quad&&\text{ for almost-all }t\in (0,T),
 \label{eq:conv-mi}
 \\
 \partial_t\mu_\eps^t(\psi_i)\,&\to\, m_i'\quad&&\text{ as Radon measures on
 }(0,T). \label{eq:conv-mi-2} 
\end{alignat}
Let $S$ denote the countable set of times $t\in (0,T)$ where for some $i\in\N$ the
measure $m_i'$ has an atomic part in $t$. We claim that
\eqref{eq:conv-mi} holds on $(0,T)\setminus S$. To see this
we choose a point $t\in (0,T)\setminus S$
and a sequence of points $(t_j)_{j\in\N}$ in $(0,T)\setminus S$, such that
$t_j\nearrow t$ and \eqref{eq:conv-mi} holds for all $t_j$. We then obtain
\begin{alignat}{2}
  \lim_{j\to\infty} m_i'([t_j,t])\,&=\,0 &&\text{ for all }i\in\N,
  \label{eq:atom-mi}\\ 
  \lim_{\eps\to 0} \partial_t \mu^\eps(\psi_i)([t_j,t])\,&=
  m_i'([t_j,t])\quad&&\text{ for all }i,j\in\N. \label{eq:conv-d-mi}
\end{alignat}
Moreover
\begin{align*}
  |m_i(t)-\mu_\eps^t(\psi_i)|\,&\leq\, 
  |m_i(t)-m_i(t_j)|+|m_i(t_j)-\mu_\eps^{t_j}(\psi_i)|
  +|\mu_\eps^{t_j}(\psi_i)- \mu_\eps^t(\psi_i)|\\
  &\leq\, |m_i'([t_j,t])| + |m_i(t_j)-\mu_\eps^{t_j}(\psi_i)| +
  |\partial_t \mu_\eps^t(\psi_i)([t_j,t])|
\end{align*}
Taking first $\eps\to 0$ and then $t_j\nearrow t$ we deduce by
  \eqref{eq:atom-mi}, \eqref{eq:conv-d-mi} that \eqref{eq:conv-mi} holds for all
$i\in\N$ and all $t\in (0,T)\setminus S$.

Taking now an arbitrary $t\in (0,T)$ such that \eqref{eq:conv-mi} holds,
by \eqref{ass:bound-mu} there exists a subsequence $\eps\to 0$ such that
\begin{gather}
  \mu_\eps^t\,\to\, \mu^t \quad\text{ as Radon-measures on
  }\Omega. \label{eq:conv-mu-eps-t} 
\end{gather}
We deduce that $\mu^t(\psi_i)=m_i(t)$ and since $(\psi_i)_{i\in\N}$ is
dense in $C^0(\overline{\Omega})$ we can identify any limits of
$(\mu_\eps^t)_{\eps>0}$ and obtain \eqref{eq:conv-mu-eps-t} for the
whole sequence selected in \eqref{eq:conv-mi},
 \eqref{eq:conv-mi-2} and for all $t\in (0,T)$, for which
\eqref{eq:conv-mi} holds. Moreover for any $\psi\in
C^0(\overline{\Omega})$ the map $t\mapsto \mu_\eps^t(\psi)$ has no
jumps in $(0,T)\setminus S$. After possibly taking another subsequence we
can also ensure that as $\eps\to 0$
\begin{gather*}
  \mu_\eps^0\,\to\, \mu^0,\qquad \mu_\eps^T\,\to\, \mu^T
\end{gather*}
as Radon measures on $\overline{\Omega}$. This proves \eqref{eq:conv-mu-t}.

By the Dominated Convergence Theorem we conclude that for any $\eta\in
C^0(\overline{\Omega_T})$ 
\begin{gather*}
  \int_{\Omega_T} \eta\,d\mu\,=\, \lim_{\eps\to 0}
  \int_{\Omega_T} \eta\,d\mu_\eps \,=\,  \lim_{\eps\to 0}\int_0^T
  \int_\Omega \eta(t,x)\,d\mu_\eps^t(x)\,dt\,=\, \int_0^T \int_\Omega
  \eta(t,x)\,d\mu^t(x)\,dt,
\end{gather*}
which implies \eqref{eq:disint-mu}.

By \eqref{eq:bv-mu}, the $L^1(0,T)$-compactness of sequences that are
uniformly bounded in $BV(0,T)$, the lower-semicontinuity of the
$BV$-norm under $L^1$-convergence, and \eqref{eq:conv-mu-t} we conclude
that \eqref{eq:mu-bv} holds. 
\end{proof}

\begin{proof}[Proof of Theorem \ref{the:lsc-H}]
For almost all $t\in (0,T)$ we obtain from Fatou's Lemma and
\eqref{ass:bound-squares}, \eqref{ass:bound-mu} that
\begin{gather}
  \liminf_{\eps\to 0} \Big(\mu_\eps^t(\Omega) +
  \int_\Omega
  \frac{1}{\eps}w_\eps^2(t,x)\,dx\Big)\,<\,\infty.
  \label{eq:bound-fatou}  
\end{gather}
Let $S\subset (0,T)$ be as in Proposition \ref{prop:KRT} and fix a $t\in
(0,T)\setminus S$ such that \eqref{eq:bound-fatou} 
holds. 
Then we deduce from \cite[Theorem 4.1, Theorem 5.1]{RS} and
\eqref{eq:conv-mu-t} that
\begin{align*}
  &\frac{1}{c_0}\mu^t \text{ is an integral }(n-1)\text{-varifold},\\
  &\mu^t\,\geq\, \frac{c_0}{2} |\nabla u(t,\cdot)|,
\end{align*}
and that $\mu^t$ has weak mean curvature ${H}(t,\cdot)$
satisfying
\begin{gather}
  \int_\Omega |{H}(t,x)|^2\,d\mu^t(x)\,\leq\, \liminf_{\eps\to 0}
  \frac{1}{\eps}\int_\Omega w_\eps(t,x)^2\,dx. \label{eq:lsc-H-t}
\end{gather}
By \eqref{eq:lsc-H-t} and Fatou's Lemma 
we
obtain that 
\begin{align*}
  \int_0^T \int_\Omega |{H}(t,x)|^2\,d\mu^t(x)\,dt \,&\leq\,
  \int_0^T \left(\liminf_{\eps\to 0}
  \frac{1}{\eps}\int_\Omega w_\eps(t,x)^2\,dx\right)\,dt \\
  \,&\leq\, \liminf_{\eps\to 0} \frac{1}{\eps}\int_{\Omega_T}
  w_\eps^2\,dxdt,
\end{align*}
which proves \eqref{eq:lsc-H}.

For later use we also associate general varifolds to $\mu_\eps^t$ and
consider their convergence as $\eps\to 0$. Let
$\nu_\eps(t,\cdot):\Omega\to S^{n-1}_1(0)$ be an extension of $\nabla 
u_\eps(t,\cdot)/{|\nabla u_\eps(t,\cdot)|}$ to 
the set $\{\nabla u_\eps(t,\cdot)=0\}$. Define the
projections $P_\eps(t,x)\,:=\, Id - \nu_\eps(t,x)\otimes\nu_\eps(t,x)$
and consider the general varifolds $V_\eps^t$ and the integer
rectifiable varifold $c_0^{-1}V^t$ defined by
\begin{align}
  V_\eps^t(f)\,&:=\, \int_{\Omega}
  f(x,P_\eps(t,x))\,d\mu_\eps^t(x), \label{eq:def-V-eps-t}\\
  V^t(f)\,&:=\, \int_{\Omega}
  f(x,P(t,x))\,d\mu^t(x)\label{eq:def-V-t}
\end{align}
for $f\in C^0_c(\Omega\times\R^{n\times n})$, where
$P(t,x)\in\R^{n\times n}$ denotes the projection onto the 
tangential plane $T_x\mu^t$. Then we deduce from the proof of
\cite[Theorem 4.1]{RS} that
\begin{gather}
  V_\eps^t\,\to\, V^t\quad\text{ as }\eps\to 0 \label{eq:conv-V-t}
\end{gather}
in the sense of varifolds.
\end{proof}
\section{Proof of Theorem \ref{the:lsc-v}}
\label{sec:proof-main}
\subsection{Equipartition of energy}
We start with a preliminary result, showing the important
\emph{equipartition of  energy}: the \emph{discrepancy measure}
\begin{gather}
  \xi_\eps\,:=\,
  \Big(\frac{\eps}{2}|\nabla u_\eps|^2-\frac{1}{\eps}W(u_\eps)\Big)\LL^{n+1}
\label{eq:def-xi-eps}
\end{gather}
vanishes in the limit $\eps\to 0$.

To prove this we combine results from \cite{RS}  with
a refined version of Lebesgue's dominated convergence Theorem
\cite{PS}, see also \cite[Lemma 4.2]{Ro}.  
\begin{proposition}
\label{prop:xi=0}
For a subsequence $\eps\to 0$ we obtain that
\begin{gather}
  |\xi_\eps|\,\to\, 0\quad\text{ as Radon measures on
   }\overline{\Omega_T}. \label{eq:zero-xi}
\end{gather}
\end{proposition}
\begin{proof}
Let us define the measures
\begin{align*}
  \xi_\eps^t\,&:=\,
  \Big(\frac{\eps}{2}|\nabla
  u_\eps|^2-\frac{1}{\eps}W(u_\eps)\Big)(t,\cdot)\LL^{n}
\end{align*}
on $\Omega$.
For $\eps>0$, $k\in\N$, we define the sets
\begin{gather}
  \mathcal{B}_{\eps,k}\,:=\, \{t\in (0,T) :
  \int_\Omega\frac{1}{\eps}w_\eps^2(t,x)\,dx\,>\,k\}. \label{eq:def-BEk}
\end{gather}
We then obtain from \eqref{ass:bound-squares} that
\begin{gather}
  \Lsquares \,\geq\, \int_0^T
  \int_\Omega\frac{1}{\eps}w_\eps^2(t,x)\,dxdt\,\geq\,
  |\mathcal{B}_{\eps,k}|k. \label{eq:bound-BEk}
\end{gather}
Next we define the (signed) Radon-measures $\xi_{\eps,k}^t$ by
\begin{gather}
  \xi_{\eps,k}^t \,:=\,
  \begin{cases}
    \xi_\eps^t &\text{ for }t\in (0,T)\setminus \mathcal{B}_{\eps,k},\\
    0 &\text{ for }t\in \mathcal{B}_{\eps,k}.
  \end{cases}
  \label{eq:def-xi-k}
\end{gather}
By \cite[Proposition 4.9]{RS}, we have
\begin{gather}
  |\xi_{\eps_j}^t| \,\to\, 0 \ (j\to\infty) \text{ as Radon measures on }\Omega
  \label{eq:conv-xi-t} 
\end{gather}
for any subsequence $\eps_j\to 0$ $(j\to\infty)$ such that
\begin{gather*}
  \limsup_{j\to \infty} \int_{\Omega}\frac{1}{\eps_j}w_{\eps_j}^2(t,x)\,dx
  \,<\,\infty.
\end{gather*}
By \eqref{ass:bound-mu}, \eqref{eq:def-xi-k} we deduce that for any $\eta\in
C^0(\overline{\Omega_T},\R^+_0)$, $k\in\N$, and almost all $t\in (0,T)$ 
\begin{gather}
  |\xi_{\eps,k}^t|(\eta(t,\cdot))\,\to\, 0\quad\text{ as }\eps\to 0
  \label{eq:conv-xi-k-t}
\end{gather}
and that
\begin{gather}
  |\xi_{\eps,k}^t|(\eta(t,\cdot))\,=\,
  \big(1-\Chi_{\mathcal{B}_{\eps,k}}(t)\big)|\xi_\eps^t|\big(\eta(t,\cdot)\big)\,\leq\,  
  \Lenergy\|\eta\|_{C^0(\overline{\Omega_T})}. \label{eq:dom-xi-k} 
\end{gather}
By the Dominated Convergence Theorem, \eqref{eq:conv-xi-k-t} and
\eqref{eq:dom-xi-k} imply that
\begin{gather}
  \int_0^T |\xi_{\eps,k}^t|(\eta(t,\cdot))\,dt \,\to\, 0\quad\text{ as }\eps\to
  0. \label{eq:conv-xi-k}
\end{gather}
Further we obtain that
\begin{align}
  \int_0^T |\xi_\eps^t|(\eta(t,\cdot))\,dt\,\leq\,&
  \int_0^T |\xi_{\eps,k}^t|(\eta(t,\cdot))\,dt +
  \int_{\mathcal{B}_{\eps,k}}
  |\xi_{\eps}^t|(\eta(t,\cdot))\,dt \notag\\
  \leq\,&\int_0^T|\xi_{\eps,k}^t|(\eta(t,\cdot))\,dt +
  \int_{\mathcal{B}_{\eps,k}}
  \mu_{\eps}^t(\eta(t,\cdot))\,dt. \label{eq:est-lim-xi} 
\end{align}
For $k\in\N$ fixed we deduce from \eqref{ass:bound-mu},
\eqref{eq:bound-BEk}, \eqref{eq:est-lim-xi} that 
\begin{gather}
  \limsup_{\eps\to 0}\int_0^T\int_\Omega
 |\xi_\eps^t|( \eta(t,\cdot))\,dt \,\leq\, \lim_{\eps\to 0}
  \int_0^T|\xi_{\eps,k}^t|(\eta(t,\cdot))\,dt +
  \|\eta\|_{C^0(\Omega_T)}\Lenergy\frac{\Lsquares}{k}.
\end{gather}
By \eqref{eq:conv-xi-k} and since $k\in\N$ was arbitrary this proves the
Proposition. 
\end{proof}
\subsection{Convergence of approximate velocities}
In the next step in the proof of Theorem \ref{the:lsc-v} we
define approximate velocity vectors and show their convergence as
$\eps\to 0$. 
\begin{lemma}
\label{lem:velo}
Define $v_\eps:\Omega_T\to\Rn$ by
\begin{gather}
  v_\eps\,:=\,
  \begin{cases}
    -\frac{\partial_t u_\eps}{|\nabla u_\eps|}\frac{\nabla
      u_\eps}{|\nabla u_\eps|} &\text{ if }|\nabla u_\eps|\neq 0,\\[2mm]
    0&\text{ otherwise. }
  \end{cases}
  \label{eq:def-v-eps}
\end{gather} 
Then there exists a function $v\in L^2(\mu,\Rn)$ such that
\begin{gather}
  (\eps|\nabla u_\eps|^2\,\LL^{n+1},v_\eps)\,\to\, (\mu,v)\quad\text{ as
  }\eps\to 0 \label{eq:conv-pair}
\end{gather}
in the sense of measure function pair convergence (see the Appendix
\ref{app:hutch}) and such that \eqref{eq:v-lsc} is satisfied.
\end{lemma}
\begin{proof}
We define Radon measures
\begin{gather}
  \tilde{\mu}_\eps\,:=\, \eps|\nabla u_\eps|^2\,\LL^{n+1}\,=\, \mu_\eps
  +\xi_\eps. \label{eq:def-tilde-mu}
\end{gather}
From \eqref{eq:conv-mu}, \eqref{eq:zero-xi} we deduce that
\begin{gather}
  \tilde{\mu}_\eps \,\to\, \mu\quad \text{ as Radon measures on
  }\overline{\Omega_T}. \label{eq:conv-tilde-mu} 
\end{gather}
Next we observe that $(\tilde{\mu}_\eps, v_\eps)$ is a function-measure
pair in the sense of \cite{Hu} (see also
Definition \ref{def:meas-funct-pair} in Appendix B)
and that by \eqref{ass:bound-squares}
\begin{gather}
  \int_{\Omega_T} |v_\eps|^2\,d\tilde{\mu}_\eps \,\leq\, \int_{\Omega_T}
  \eps (\partial_t u_\eps)^2\,dxdt\,\leq\,
  \Lsquares. \label{eq:bound-v-eps} 
\end{gather}
By Theorem \ref{the:hutch} we therefore deduce that there
exists a subsequence 
$\eps\to 0$ and a function $v\in L^2(\mu,\Rn)$ such that
\eqref{eq:conv-pair} and \eqref{eq:v-lsc} hold.
\end{proof}
%
%
\begin{lemma}\label{lem:perp-v}
For $\mu$-almost all $(t,x)\in\Omega_T$
\begin{gather}
  v(t,x)\,\perp\, T_x\mu^t. \label{eq:perp-v}
\end{gather}
\end{lemma}
\begin{proof}
We follow \cite[Proposition 3.2]{Mose01}. Let $\nu_\eps:\Omega_T\to
S^{n-1}_1(0)$ be an extension of $\nabla u_\eps/{|\nabla u_\eps|}$ to
the set $\{\nabla u_\eps=0\}$ and define 
projection-valued maps $P_\eps:\Omega_T\,\to\, \R^{n\times n}$,
\begin{gather*}
  P_\eps\,:=\, Id - \nu_\eps\otimes\nu_\eps.
\end{gather*}
Consider next the general varifolds $\tilde{V}_\eps,V$ defined by
\begin{align}
  \tilde{V}_\eps(f)\,&:=\, \int_{\Omega_T}
  f(t,x,P_\eps(t,x))\,d\tilde{\mu}_\eps(t,x), \label{eq:def-V-eps} \\
  V(f)\,&:=\, \int_{\Omega_T}
  f(t,x,P(t,x))\,d\mu^t(x) \label{eq:def-V} 
\end{align}
for $f\in C^0_c(\Omega_T\times\R^{n\times n})$, where
$P(t,x)\in\R^{n\times n}$ denotes the projection onto the 
tangential plane $T_x\mu^t$.

From \eqref{eq:conv-V-t}, Proposition
\ref{prop:xi=0}, and Lebesgue's Dominated Convergence
Theorem we deduce that
\begin{gather}
  \lim_{\eps\to 0} \tilde{V}_\eps \,=\, V \label{eq:conv-V}
\end{gather}
as Radon-measures on $\Omega_T\times\R^{n\times n}$.
%

Next we define functions $\hat{v}_\eps$ on $\Omega_T\times\R^{n\times n}$ 
by
\begin{gather*}
  \hat{v}_\eps(t,x,Y)\,=\, v_\eps(t,x)\quad\text{ for all
  }(t,x)\in\Omega_T,\, Y\in\R^{n\times n}.
\end{gather*}
We then observe that
\begin{gather*}
  \int_{\Omega_T\times\R^{n\times n}} \hat{v}_\eps^2\,dV_\eps\,=\,
  \int_{\Omega_T} v_\eps^2\,d\tilde{\mu}_\eps\,\leq\, \Lsquares
\end{gather*}
and deduce from \eqref{eq:conv-V} and Theorem \ref{the:hutch} the
existence of $\hat{v}\in L^2(V,\Rn)$ such that $(V_\eps,\hat{v}_\eps)$
converge to $(V,\hat{v})$ as measure-function pairs on
$\Omega_T\times\R^{n\times n}$ with values in $\Rn$.

We consider now $h\in
C^0_c(\R^{n\times n})$ such that $h(Y)=1$ for all projections $Y$.
We deduce 
that for any $\eta\in C^0_c(\Omega_T,\Rn)$
\begin{align*}
  \int_{\Omega_T} \eta\cdot v\,d\mu\,&=\, \lim_{\eps\to 0}
  \int_{\Omega_T\times\R^{n\times n}} \eta(t,x)\cdot h(Y)\hat{v}_\eps(t,x,Y)\,
  dV_\eps(t,x,Y)\\ 
  &= \int_{\Omega_T}
  \eta(t,x)\cdot\hat{v}(t,x,P(t,x))\,d\mu(t,x),
\end{align*}
which shows that for $\mu$-almost all $(t,x)\in\Omega_T$
\begin{gather}
  \hat{v}(t,x,P(t,x))\,=\, v(t,x). \label{eq:hat-v}
\end{gather}
Finally we observe that for $h,\eta$ as above
\begin{align*}
  &\int_{\Omega_T} \eta(t,x)\cdot P(t,x)v(t,x)\,d\mu(t,x)\\
  =\,&\int_{\Omega_T\times\R^{n\times n}} \eta(t,x)h(Y)\cdot Y
  \hat{v}(t,x,Y)\,dV(t,x,Y)\\ 
  =\,&\lim_{\eps\to 0}
  \int_{\Omega_T\times\R^{n\times n}}
  \eta(t,x)h(Y)\cdot Y\hat{v}_\eps(t,x,Y)\,dV_\eps(t,x,Y)\\ 
  =\,&\lim_{\eps\to 0}\int_{\Omega_T} \eta(t,x)\cdot
  P_\eps(t,x)v_\eps(t,x)\,d\tilde{\mu}_\eps(t,x)\,=\, 0
\end{align*}
since $P_\eps v_\eps=0$. This shows that $P(t,x)v(t,x)=0$ for
$\mu$-almost all $(t,x)\in\Omega_T$.
\end{proof}
\begin{proof}[Proof of Theorem \ref{the:lsc-v}]
By \eqref{ass:bound-squares} there exists a subsequence $\eps\to 0$ and
a Radon measure $\beta$ on $\overline{\Omega}_T$ 
such that 
\begin{gather}
  \big(\eps(\partial_t u_\eps)^2 +
  \frac{1}{\eps}w_\eps^2\big)\LL^{n+1}\,\to\, \beta,\qquad
  \beta(\overline{\Omega_T})\,\leq\, \Lsquares. \label{eq:conv-beta}
\end{gather}
Using \eqref{eq:neumann} we compute that for any $\eta\in
C^1_c((0,T)\times\overline{\Omega})$ 
\begin{equation}
 \begin{split}
  \int_{\Omega_T} \eta\,d\alpha_\eps\,=&\, \int_{\Omega_T}\eta
  \big(\eps(\partial_t u_\eps)^2 + 
  \frac{1}{\eps}w_\eps^2\big)\,dxdt -2\int_{\Omega_T}
  \partial_t\eta\,d\mu_\eps 
  \\
  +& 2\int_{\Omega_T}
  \eps\nabla\eta\cdot\partial_t u_\eps\nabla
  u_\eps\,dxdt\,. \label{eq:KRT-1-rep}
\end{split}
\end{equation}
As $\eps$ tends to zero the term on the left-hand side and the first two
terms on the right-hand-side converge by \eqref{eq:conv-mu},
\eqref{eq:conv-alpha} and \eqref{eq:conv-beta}.
For the third term on the right-hand side of \eqref{eq:KRT-1-rep}
we obtain from \eqref{eq:conv-pair} that
\begin{align*}
  \lim_{\eps\to 0}\int_{\Omega_T}
  \nabla\eta \cdot\eps \partial_t u_\eps\nabla u_\eps\,dxdt
  \,&=\, -\lim_{\eps\to 0}\int_{\Omega_T}
  \nabla\eta \cdot v_\eps \,\eps |\nabla u_\eps|^2\,dxdt\notag
  \\
  \,&=\,- \int_{\Omega_T}
  \nabla\eta \cdot v\,d\mu.
\end{align*}
Therefore, taking $\eps\to 0$ in \eqref{eq:KRT-1-rep} we deduce that
\begin{gather*}
  \int_{\Omega_T} \eta\,d\alpha\,=\, \int_{\Omega_T} \eta\,d\beta -2
  \int_{\Omega_T} \partial_t\eta\,d\mu - 2\int_{\Omega_T}
  \nabla\eta\cdot v \,d\mu
\end{gather*}
holds for all $\eta\in C^1_c((0,T)\times\overline{\Omega})$.
This yields that
\begin{gather*}
  \Big| \int_{\Omega_T} \partial_t\eta + \nabla\eta\cdot
  v\,d\mu\Big|\,\leq\,
  \|\eta\|_{C^0(\overline{\Omega_T})}\frac{1}{2}\big(\alpha(\overline{\Omega_T}) +
  \beta(\overline{\Omega_T})\big),
\end{gather*}
which shows together with \eqref{eq:perp-v} that $v$ is a generalized
velocity vector for $(\mu^t)_{t\in (0,T)}$ in the sense of Definition
\ref{def:gen-velo}. The estimate \eqref{eq:v-lsc} was already proved in
Lemma \ref{lem:velo}. 
\end{proof}
\section{Proof of Theorem \ref{the:main}}
We start with the convergence of a `diffuse mean
curvature term'.
\begin{lemma}\label{lem:first-var}
Define
\begin{gather*}
  H_\eps\,:=\, \frac{1}{\eps}w_\eps\frac{\nabla u_\eps}{|\nabla
  u_\eps|^2},
\end{gather*}
let $\tilde{\mu}_\eps=\eps|\nabla u_\eps|^2\,\LL^{n+1}$, and let
$v_\eps,\, v$ be as in \eqref{eq:def-v-eps}, \eqref{eq:conv-pair}. 
Then 
\begin{align}
  (\tilde{\mu}_\eps,H_\eps)\,&\to\, (\mu,H), \label{eq:conv-H-eps}\\
  (\tilde{\mu}_\eps,v_\eps-H_\eps)\,&\to\, (\mu,v-H)
  \label{eq:conv-v+H}
\end{align}
as $\eps\to 0$ in the sense of measure function pair convergence.
In particular
\begin{gather}
 \int_{\Omega_T} \eta |v-H|^2\,d\mu \,\leq\,
 \alpha(\eta) \label{eq:est-v+H} 
\end{gather}
holds for all $\eta\in C^0(\overline{\Omega_T},\R^+_0)$.
\end{lemma}
\begin{proof}
We use similar arguments as in the proof of Proposition \ref{prop:xi=0}.
For $\eps>0$, $k\in\N$, we define sets
\begin{gather}
  \mathcal{B}_{\eps,k}\,:=\, \{t\in (0,T) :
  \int_\Omega \frac{1}{\eps}w_\eps(t,x)^2\,dx\,>\,k\}. \label{eq:def-BEk-2}
\end{gather}
We then obtain from \eqref{ass:bound-squares} that
\begin{gather}
  \Lsquares\,\geq\, \int_{\Omega_T}
  \frac{1}{\eps}w_\eps^2\,dxdt\,\geq\,
  |\mathcal{B}_{\eps,k}|k. \label{eq:bound-BEk-2}
\end{gather}
Next we define functionals $T_{\eps,k}^t\in C^0_c(\Omega,\Rn)^*$ by
\begin{gather}
  T_{\eps,k}^t(\psi) \,:=\,
  \begin{cases}
    \int_\Omega \psi(x)\cdot w_\eps(t,x)\nabla u_\eps(t,x)\,dx &\text{ for }t\in
    (0,T)\setminus \mathcal{B}_{\eps,k},\\ 
    \int_\Omega \psi(x)\cdot H(t,x)\,d\mu^t(x) &\text{ for }t\in \mathcal{B}_{\eps,k}.
  \end{cases}
  \label{eq:def-xi-k-2}
\end{gather}
Considering the general $(n-1)$-varifolds $V_\eps^t,V^t$ defined in
\eqref{eq:def-V-eps-t}, \eqref{eq:def-V-t} we obtain from
\cite[Proposition 4.10]{RS} and \eqref{eq:conv-V-t} that
\begin{equation}
\begin{split}
  \lim_{j\to\infty} \int_\Omega \psi\cdot w_{\eps_j}(t,x)\nabla
  u_{\eps_j}(t,x)\,dx 
  \,=&\, -\lim_{j\to\infty} \delta V_{\eps_j}^t(\psi)\,
  \\
  =&\, -\delta\mu^t(\psi)
  \,=\, \int_\Omega \psi\cdot H(t,x)\,d\mu^t(x) \label{eq:conv-xi-t-2} 
\end{split}
\end{equation}
for any subsequence $\eps_j\to 0$ $(j\to\infty)$ such that
\begin{gather*}
  \limsup_{j\to \infty} \int_{\Omega}
  \frac{1}{\eps_j}w_{\eps_j}^2\,dxdt\,<\,\infty.
\end{gather*}
Therefore we deduce from \eqref{eq:def-xi-k-2}, \eqref{eq:conv-xi-t-2}
that for all $\eta\in C^0_c(\Omega_T,\Rn)$, $k\in\N$, and almost all
$t\in (0,T)$  
\begin{gather}
  T_{\eps,k}^t(\eta(t,\cdot))\,\to\, \int_\Omega \eta(t,x)\cdot
  H(t,x)\,d\mu^t(x)\quad\text{ as }\eps\to 0 
  \label{eq:conv-xi-k-t-2}
\end{gather}
and that
\begin{align}
  &\big|T_{\eps,k}^t(\eta(t,\cdot))\big|\notag\\
  \leq\,&
  \big(1-\Chi_{\mathcal{B}_{\eps,k}}(t)\big)\left\vert\int_\Omega
  \eta(t,x)\cdot w_\eps(t,x)\nabla
  u_\eps(t,x)\,dx\right\vert\notag\\
  & +\Chi_{\mathcal{B}_{\eps,k}}(t)\left\vert\int_\Omega \eta(t,x)\cdot
  H(t,x)\,d\mu^t(x)\right\vert\notag\\ 
  \leq\,&
  \|\eta\|_{C^0(\Omega_T)}\big(1-\Chi_{\mathcal{B}_{\eps,k}}(t)\big)\Big(\int_\Omega
  \frac{1}{2\eps} 
  w_\eps(t,x)^2\,dx\Big)^{1/2} \Big(\int_\Omega
  \frac{\eps}{2}|\nabla u_\eps(t,x)|^2\,dx\Big)^{1/2}\notag\\
  &+\int_\Omega |\eta(t,x)||
  H(t,x)|\,d\mu^t(x)\notag\\
  \leq\,&
  \|\eta\|_{C^0(\Omega_T)}\sqrt{\frac{k}{2}}\sqrt{\Lenergy}
  +\int_\Omega |\eta(t,x)||
  H(t,x)|\,d\mu^t(x),\label{eq:dom-xi-k-2} 
\end{align}
where the right-hand side is bounded in $L^1(0,T)$, uniformly with
respect to $\eps>0$.

By the Dominated Convergence Theorem, \eqref{eq:conv-xi-k-t-2} and
\eqref{eq:dom-xi-k-2} imply that
\begin{gather}
  \int_0^T T_{\eps,k}^t(\eta(t,\cdot))\,dt \,\to\, \int_{\Omega_T}
  \eta\cdot H \,d\mu\quad\text{ as }\eps\to
  0. \label{eq:conv-xi-k-2}
\end{gather}
Further we obtain that
\begin{align}
  &\Big|\int_{\Omega_T} \eta\cdot w_\eps\nabla u_\eps\,dxdt -
  \int_{\Omega_T} \eta\cdot H\,d\mu\Big|\notag\\
  \leq\,&
  \Big|\int_0^T T_{\eps,k}^t(\eta(t,\cdot))\,dt -
  \int_{\Omega_T} \eta\cdot H\,d\mu\Big|\notag\\
 +\, & \Big|\int_{\mathcal{B}_{\eps,k}}\int_\Omega
  \eta(t,x)\cdot H(t,x)\,d\mu^t(x) dt\Big| +
  \Big|\int_{\mathcal{B}_{\eps,k}}\int_\Omega 
  \eta\cdot w_\eps\nabla
  u_\eps\,dx\,dt\Big| \label{eq:est-lim-T}
\end{align}
The last term on the right-hand side we further estimate by
\begin{align}
  &\Big|\int_{\mathcal{B}_{\eps,k}}\int_\Omega 
  \eta(t,x)\cdot w_\eps(t,x)\nabla
  u_\eps(t,x)\,dx\,dt\Big|\notag\\
  \leq\,& \|\eta\|_{C^0(\Omega_T)}
  \Big(\int_{\Omega_T}
  \frac{1}{2\eps}w_\eps^2\,dxdt\Big)^{1/2}
  |\mathcal{B}_{\eps,k}|^{1/2}\sqrt{\Lenergy} 
  \notag\\ 
  \leq\, &\|\eta\|_{C^0(\Omega_T)} \Lsquares
  \frac{1}{\sqrt{k}}\sqrt{\Lenergy}, \label{eq:est-lim-T-3}
\end{align}
where we have used \eqref{ass:bound-mu} and \eqref{eq:bound-BEk-2}. 
For the second term on the right-hand side of \eqref{eq:est-lim-T} we
obtain 
\begin{align}
  \Big|\int_{\mathcal{B}_{\eps,k}}\int_\Omega
  \eta(t,x)\cdot H(t,x)\,d\mu^t(x) dt\Big|\,&\leq\,
  |\mathcal{B}_{\eps,k}|^{1/2}
  \|\eta\|_{C^0(\Omega_T)}^{1/2}\Big(\int_{\spt(\eta)}  
  H^2\,d\mu\Big)^{1/2} \notag\\
  &\leq\,
  \frac{\sqrt{\Lsquares}}{\sqrt{k}}\|\eta\|_{C^0(\Omega_T)}^{1/2}\sqrt{\Lsquares}, 
  \label{eq:est-lim-T-2} 
\end{align} 
where we have used \eqref{eq:lsc-H} and \eqref{ass:bound-squares}.
Finally, for $k\in\N$ fixed, by \eqref{eq:conv-xi-k-2} we deduce that
\begin{gather}
  \lim_{\eps\to 0}\Big|\int_0^T T_{\eps,k}^t(\eta(t,\cdot))\,dt -
  \int_{\Omega_T} \eta\cdot H\,d\mu\Big|\,=\,0. \label{eq:est-lim-T-1}
\end{gather}
Taking $\eps\to 0$ in \eqref{eq:est-lim-T} we obtain by
\eqref{eq:est-lim-T-3}-\eqref{eq:est-lim-T-1} that
\begin{align}
  &\lim_{\eps\to 0}\Big|\int_{\Omega_T} \eta\cdot w_{\eps}\nabla u_{\eps}\,dxdt -
  \int_{\Omega_T} \eta\cdot H\,d\mu\Big|\notag\\
  \leq\,& 
  \frac{\Lsquares}{\sqrt{k}}
  \|\eta\|_{C^0(\Omega_T)} \sqrt{\Lenergy} +  \frac{1}{\sqrt{k}}\Lsquares
\end{align}
for any $k\in\N$, which proves \eqref{eq:conv-H-eps}. Using
\eqref{eq:conv-pair} this implies \eqref{eq:conv-v+H}. Finally we fix an
arbitrary nonnegative $\eta\in C^0(\overline{\Omega_T})$ and deduce that
the measure-function pair
$(\tilde{\mu}_\eps,\sqrt{\eta}(v_\eps-H_\eps))$ converges to
$(\mu,\sqrt{\eta}(v-H))$. The estimate
\eqref{eq:est-v+H} then follows from Theorem \ref{the:hutch}. 
\end{proof}

Let $\proj: [0,T]\times\overline{\Omega}\to [0,T]$ denote the projection onto
the first component and $\proj_\#$ the pushforward of measures by $\proj$. For
$\psi\in C^0(\overline{\Omega})$ we consider the measures
\begin{gather*}
  \alpha_\psi \,:=\,\proj_\# \big(\psi\alpha\big),
\end{gather*}
on $[0,T]$, that means
\begin{align*}
  \alpha_\psi(\zeta)\,&:=\, \int_{\Omega_T} \zeta(t)\psi(x)\,d\alpha(t,x),
\end{align*}
for $\zeta\in C^0([0,T])$, and set
\begin{align*}
  \alpha_\Omega \,:=\, \proj_\#\alpha.
\end{align*}
We then can estimate the atomic part of $\alpha_\Omega$ in terms of the
nucleation cost.
\begin{lemma}
\label{lem:atomic}
Let  $\A_{nuc}(\mu)$ be the nucleation cost defined in
\eqref{eq:def-S-nuc}. Then
\begin{equation}
  (\alpha_\Omega)_{atomic}[0,T]\,\geq\, 4 \A_{nuc}(\mu). \label{eq:est-A-atoms}
\end{equation}
\end{lemma}
\begin{proof}
Let $\eta\in C^1(\overline{\Omega_{T}},\R^+_0)$ be nonnegative. We
compute that 
\begin{align}
  \int_{\Omega_{T}} \eta d\alpha_\eps
  \,=\, 
  &\int_{\Omega_T} \eta \Big(\eps(\partial_t u_\eps)^2 + \frac{1}{\eps}
  w_\eps^2 + 2\partial_t u_\eps w_\eps\Big)\,dxdt\notag\\
  \,\geq\,& 4\int_{\Omega_T} \eta \partial_t u_\eps w_\eps\,dxdt\notag\\
  \,=\, & - 4\int_{\Omega_T} \partial_t\eta\,  d\mu_\eps
  + 4\int_{\Omega_T} \nabla\eta\cdot\eps\partial_t u_\eps\nabla
  u_\eps\,dxdt\notag\\
  & + 4\mu_\eps^T(\eta(T,\cdot)) -4\mu_\eps^0(\eta(0,\cdot)).\label{eq:est-act-1} 
\end{align}
Passing to the limit $\eps\to 0$ we obtain from \eqref{eq:conv-mu},
\eqref{eq:conv-mu-t}, \eqref{eq:conv-pair} that 
\begin{align}
  \int_{\Omega_T} \eta d\alpha
  \,\geq\, 
  &  -4\int_{\Omega_T} \partial_t\eta \,d\mu
  - 4\int_{\Omega_T}
  \nabla\eta \cdot v\,d\mu + 4\mu^T(\eta(T,\cdot))
  -4\mu^0(\eta(0,\cdot)). \label{eq:squares-limit}
\end{align}
We now choose  $\eta(t,x)\,=\,\zeta(t)\psi(x)$ where $\zeta\in
C^1([0,T],\R^+_0)$, $\psi\in C^1(\overline{\Omega},\R^+_0)$ in
\eqref{eq:squares-limit} and deduce that
\begin{align}
  \int_0^T \zeta d\alpha_\psi
  \,\geq\, 
  &  -4\int_0^T \partial_t\zeta \mu^t(\psi)\,dt
  + 4\int_0^T \zeta \int_\Omega \nabla\psi \cdot v(t,x)\,d\mu^t(x)\,dt\notag\\
  & + 4\zeta(T)\mu^T(\psi)
  -4\zeta(0)\mu^0(\psi). \label{eq:squares-limit-2}
\end{align}
This shows that
\begin{align}
  \alpha_\psi \,\geq\,\, &4\partial_t(\mu^t(\psi)) +4
  \Big(\int_\Omega \nabla\psi(x) \cdot v(t,x)\,d\mu^t(x)\Big)\LL^1\notag\\ 
  +\,\,& 4\big(\mu^T(\psi)-\lim_{t\uparrow T}\mu^t(\psi)\big)\delta_T
  + 4\big(\lim_{t\downarrow 0}\mu^t(\psi)-\mu^0(\psi)\big)\delta_0. \label{eq:est-alpha-psi}
\end{align}
Evaluating the atomic parts we obtain that
for any $0< t_0< T$ 
\begin{gather*}
  \alpha_\psi(\{t_0\}) \,\geq \,
  4\partial_t(\mu^t(\psi))(\{t_0\}),
\end{gather*} 
which implies that
\begin{gather}
  \alpha_\Omega(\{t_0\})\,\geq\,
  4\sup_{\psi}\partial_t(\mu^t(\psi))(\{t_0\}).
  \label{eq:Snuc-1} 
\end{gather}
where the supremum is taken over all $\psi\in C^1(\overline{\Omega})$ with
$0\leq\psi\leq 1$.

Moreover we deduce from \eqref{eq:est-alpha-psi}
\begin{gather}
  \alpha_\Omega(\{0\})\,\geq\, 4 \sup_\psi\big(\lim_{t\downarrow
    0}\mu^t(\psi)-\mu^0(\psi)\big), \label{eq:Snuc-2}\\
  \alpha_\Omega(\{T\})\,\geq\, 4 \sup_\psi\big(\mu^T(\psi)-\lim_{t\uparrow
    T}\mu^t(\psi)\big),  \label{eq:Snuc-3}
\end{gather}
where the supremum is taken over $\psi\in C(\overline{\Omega})$ with
$0\leq \psi\leq 1$. By \eqref{eq:Snuc-1}-\eqref{eq:Snuc-3} we conclude
that \eqref{eq:est-A-atoms} holds. 
\end{proof}
\begin{proof}[Proof of Theorem \ref{the:main}]
By \eqref{eq:est-v+H} we deduce that $\alpha\geq
|v-H|^2\mu$. Since $\mu=\LL^1\otimes\mu^t$ we deduce from
the Radon-Nikodym Theorem that
\begin{align}
  (\alpha_\Omega)_{ac}[0,T]\,&\geq\, \int_{\Omega_T} |v-H|^2\,d\mu,
  \label{eq:proof-A-1}
\end{align}
and from \eqref{eq:est-A-atoms} that
\begin{align}
  (\alpha_\Omega)_{atomic}[0,T]\,&\geq\, 4\A_{nuc}(\mu),
  \label{eq:proof-A-2} 
\end{align}
where $(\alpha_\Omega)_{ac}$ and $(\alpha_\Omega)_{atomic}$ denote the
absolutely continuous and
atomic part with respect to $\LL^1$ of the measure $\alpha_\Omega$.
Adding the two estimates and recalling \eqref{eq:ass-liminf} we obtain
\eqref{eq:lsc-S}.  
\end{proof}
\section{Proofs of Proposition \ref{prop:uni_gen_vel_perp} and
  Proposition \ref{prop:velo}} 
\label{sec:proofs-velo}
Define for $r>0$, $(t_0,x_0)\in\Omega_T$ the cylinders
\begin{gather*}
  Q_r(t_0,x_0)\,:=\, (t_0-r,t_0+r)\times B_r^{n}(x_0).
\end{gather*}
\begin{proof}[Proof of Proposition \ref{prop:uni_gen_vel_perp}] 
Define
\begin{gather}
  \Sigma_n(\mu)\,:=\, \big\{(t,x)\in \Omega_T : \text{ the tangential
  plane of $\mu$ in $(t,x)$ exists}\big\} \label{eq:def-Sigma-n}
\end{gather}
and choose $(t_0,x_0)\in \Sigma_n(\mu)$ such that 
\begin{gather}
  v\text{ is approximately continuous with respect to $\mu$ in
  }(t_0,x_0). \label{eq:ass-v} 
\end{gather}
Since $v\in L^2(\mu)$ we deduce
from \cite[Theorem 2.9.13]{Fede69} that 
\eqref{eq:ass-v} holds $\mu$-almost everywhere. Let
\begin{gather}
  P_0\,:=\, T_{(t_0,x_0)}\mu,\qquad
  \theta_0\,>\,0 \label{eq:def-P0}
\end{gather}
denote the tangential plane and multiplicity at $(t_0,x_0)$
respectively, and define
for any $\varphi\in C^0_c(Q_1(0))$ the scaled functions
$\varphi_\varrho\in C^0_c(Q_\varrho(t_0,x_0))$,
\begin{gather*}
  \varphi_\varrho(t,x) \,:=\,
  \varrho^{-n}\varphi\big(\varrho^{-1}(t-t_0),\varrho^{-1}(x-x_0)\big).  
\end{gather*}
We then obtain from \eqref{eq:def-P0} that
\begin{gather}
  \int_{\Omega_T} \varphi_\varrho\,d\mu\,\to\, \theta_0
  \int_{P_0}\varphi\,d\Ha^{n}\quad\text{ as }\varrho\searrow
  0. \label{eq:conv-blow} 
\end{gather}
From \eqref{eq:def_vel_L2_flow}, the Hahn--Banach Theorem, and the Riesz
Theorem we deduce that
\begin{gather}
  \vartheta\,\in\, C^1_c(\Omega_T)^*,\qquad
  \vartheta(\eta) \,:=\, \int_{\Omega_T} \nablatx\eta\cdot
  \begin{pmatrix}
    1\\v
  \end{pmatrix}
  \,d\mu \label{eq:vartheta}
\end{gather}
can be extended to a (signed) Radon-measure on $\Omega_T$. Since by the
Radon-Nikodym Theorem $D_\mu |\vartheta|$ exists and is finite
$\mu$-almost everywhere we may assume without loss of generality that 
\begin{gather}
  D_\mu |\vartheta| (t_0,x_0)\,<\,\infty. \label{eq:bound-z0}
\end{gather}

We next fix $\eta\in C^1_c(Q_1(0))$ and compute that
\begin{gather}
  \vartheta(\varrho\eta_\varrho)\,=\, \int_{\Omega_T}
  \big(\nablatx\eta\big)_\varrho\cdot 
  \begin{pmatrix}
    1\\v
  \end{pmatrix}
  \,d\mu. \label{eq:RN-eta-1}
\end{gather}
From \eqref{eq:ass-v}, \eqref{eq:conv-blow} we deduce that 
the right-hand side converges in the limit $\varrho\to 0$,
\begin{gather}
  \lim_{\varrho\to 0}\int_{\Omega_T} \big(\nablatx\eta\big)_\varrho\cdot
  \begin{pmatrix}
    1\\v
  \end{pmatrix}
  \,d\mu\,=\,  
  \theta_0 \begin{pmatrix}
    1\\v(t_0,x_0)
  \end{pmatrix}
  \cdot \int_{P_0} \nablatx\eta
  \,d\mu.  \label{eq:velo-rhs}
\end{gather}
For the left-hand side of \eqref{eq:RN-eta-1} we deduce that
\begin{gather}
  \liminf_{\varrho\searrow 0}|\vartheta(\varrho\eta_\varrho)|\,\leq\,
  \|\eta\|_{C^0_c(Q_1(0))}\liminf_{\varrho\searrow
  0}\varrho^{-n+1}|\vartheta|(Q_\varrho(t_0,x_0)) \label{eq:velo-lhs}
\end{gather}
and observe that \eqref{eq:bound-z0} implies
\begin{align}
  \infty\,>\, \lim_{\varrho\searrow 0}
  \frac{|\vartheta|(Q_\varrho(t_0,x_0))}{\mu(Q_\varrho(t_0,x_0))} 
  \,&\geq\, \liminf_{\varrho\searrow 0}
  \varrho^{-n}|\vartheta|(Q_\varrho(t_0,x_0))
  \Big(\limsup_{\varrho\searrow
  0}\varrho^{-n}{\mu(Q_\varrho(t_0,x_0))}\Big)^{-1} \notag\\ 
  &\geq\, c\liminf_{\varrho\searrow 0}
  \varrho^{-n}|\vartheta|(Q_\varrho(t_0,x_0)),
  \label{eq:dens-bdd}
\end{align}
since by \eqref{eq:conv-blow} for any $\varphi\in C^0_c(Q_{2}(0),\R^+_0)$
with $\varphi\geq 1$ on $Q_1(0)$
\begin{gather*}
  \limsup_{\varrho\searrow
  0}\varrho^{-n}\mu(Q_\varrho(t_0,x_0))\,\leq\,
 \limsup_{\varrho\searrow
  0}\int_{\Omega_T}\varphi_\varrho\,d\mu \,\leq\, C(\varphi).
\end{gather*}
Therefore  
\eqref{eq:RN-eta-1}-\eqref{eq:dens-bdd} yield
\begin{gather}
  \theta_0 \begin{pmatrix}
    1\\v(t_0,x_0)
  \end{pmatrix}
  \cdot \int_{P_0} \nablatx\eta
  \,d\mu  \,=\, 0. \label{eq:RN-eta-3}
\end{gather}
Now we observe that the integral over the projection of $\nablatx \eta$
onto $P_0$ vanishes. This shows that 
\begin{gather}
  \int_{P_0} \nablatx \eta \,d\Ha^n \,\in\,
  P_0^\perp. \label{eq:perp} 
\end{gather}
Since $\eta$ can be chosen such that the integral in
\eqref{eq:perp} takes an arbitrary direction normal to
$P_0$ we obtain 
from \eqref{eq:RN-eta-3} that $v(t_0,x_0)$ satisfies
\eqref{eq:gen_v_perp}. If $T_{x_0}\mu^{t_0}$ exists then
\begin{gather*}
  T_{(t_0,x_0)}\mu\,=\, \Big(\{0\}\times T_{x_0}\mu^{t_0}\Big)
  \oplus\,\operatorname{span}
  \begin{pmatrix} 
    1\\v(x_0)
  \end{pmatrix}
\end{gather*}
and we obtain that $v$ is uniquely determined.
\end{proof}
To prepare the proof of Proposition \ref{prop:velo} we first show that
$\mu$ is absolutely continuous with respect to $\Ha^n$.
\begin{proposition}\label{prop:upp-dens-mu}
For any $D\subset\subset \Omega$ there exists $C(D)$ such that for
all $x_0\in D$ and almost all $t_0\in (0,T)$
\begin{gather}
  \limsup_{r\searrow 0} r^{-n}\mu\big(Q_r(t_0,x_0)\big) \,\leq\,
  C(D)\Lenergy + \liminf_{\eps\to 0} \int_D
  \frac{1}{\eps}w_\eps^2(t_0,x)\,dx. \label{eq:upp-dens-mu}  
\end{gather}
In particular,
\begin{gather}
  \limsup_{\rho\to 0}\frac{\mu(B_\rho(t_0,x_0))}{\rho^n}\,<\,\infty\quad
  \text{for }\mu-\text{almost every } (t_0,x_0)
  \label{eq:upp-dens-mu-2} 
\end{gather}
and $\mu$ is absolutely continuous with respect to $\Ha^n$,
\begin{gather}
  \mu\,<<\,\Ha^n. \label{eq:lem:ac}
\end{gather}
\end{proposition}
\begin{proof}
Let
\begin{gather*}
  r_0\,:=\, \min\Big\{ 1, \frac{1}{2}\dist(D,\partial\Omega), |t_0|,
  |T-t_0|\Big\}.
\end{gather*}
Then we obtain for all $r<r_0$, $x_0\in D$,
from \eqref{eq:zero-xi} and \cite[Proposition 4.5]{RS} that
\begin{align}
  &\frac{1}{r}\int_{t_0-r}^{t_0+r}
  r^{1-n}\mu^t\big(B_r^{n}(x_0)\big)\,dt \notag\\
  \leq\,
  &\frac{1}{r}\int_{t_0-r}^{t_0+r}
  r_0^{1-n}\mu^t\big(B_{r_0}^{n}(x_0)\big)\,dt
  + \frac{1}{4(n-1)^2}
  \frac{1}{r}\int_{t_0-r}^{t_0+r}\left( \liminf_{\eps\to 0}
  \int_D \frac{1}{\eps}w_\eps^2(t,x)\,dx\right)\,dt. \label{eq:upper-int}
\end{align}
By Fatou's Lemma and \eqref{ass:bound-squares}
\begin{gather}
  t\,\mapsto\, \liminf_{\eps\to 0}
  \int_D \frac{1}{\eps}w_\eps^2(t,x)\,dx\quad\text{ is in
  }L^1(0,T) \label{eq:l1-fatou}
\end{gather}
and by \eqref{ass:bound-mu} we deduce
for almost all $t_0\in (0,T)$ that
\begin{align*}
  &\limsup_{r\searrow 0}\frac{1}{r}\int_{t_0-r}^{t_0+r}
  r^{1-n}\mu^t\big(B_r^{n}(x_0)\big)\,dt \\
  \leq\,
  &2r_0^{1-n}\Lenergy  + \frac{1}{2(n-1)^2}
  \liminf_{\eps\to 0}\int_D \frac{1}{\eps}w_\eps^2(t_0,x)\,dx.
\end{align*}
Since $r_0$ depends only on $D,\Omega$ the inequality
\eqref{eq:upp-dens-mu} follows.

By \eqref{eq:l1-fatou} the right-hand side in \eqref{eq:upp-dens-mu} is
finite for $\LL^1$-almost all $t_0\in (0,T)$ and
$\theta^{*n}(\mu,(t,x))$ is bounded for almost all $t\in (0,T)$ and all
$x\in\Omega$. By \eqref{ass:bound-mu} we deduce that for any $I\subset
(0,T)$ with $|I|=0$
\begin{gather*}
  \mu(I\times\Omega)\,\leq\, \Lenergy |I|\,=\, 0
\end{gather*}
which implies \eqref{eq:upp-dens-mu-2}. 

To prove the final statement let $B\subset \Omega_T$ be given with
\begin{gather}
  \Ha^n(B)\,=\, 0. \label{eq:ass-nullset}
\end{gather}
Consider the family of sets $(D_k)_{k\in\N}$,
\begin{gather*}
  D_k\,:=\, \{z\in \Omega_T\,:\, \theta^{*n}(\mu,z)
   \leq k\}.
\end{gather*}
By \eqref{eq:upp-dens-mu-2}, \cite[Theorem 3.2]{Si}, and
\eqref{eq:ass-nullset} we 
obtain that for all $k\in\N$ 
\begin{gather}
  \mu(B\cap D_k)\,\leq\, 2^nk\Ha^n(B\cap D_k)\,=\, 0. \label{eq:null-k}
\end{gather}
Moreover we have that
\begin{align}
  \mu(B\setminus\bigcup_{k\in\N}D_k)\,&=\,0
  \label{eq:null-infty}
\end{align}
by \eqref{eq:upp-dens-mu-2}. By \eqref{eq:null-k},
\eqref{eq:null-infty} we conclude that
\begin{gather*}
  \mu(B)\,=\, 0,
\end{gather*}
which proves \eqref{eq:lem:ac}.
\end{proof}
To prove Proposition \ref{prop:velo} we need that $\Ha^n$-almost
everywhere on $\partial^*\{u=1\}$ the generalized tangent plane of $\mu$
exists. We first obtain the following relation between the
measures $\mu$ and $|\nablatx u|$.
\begin{proposition}\label{prop:lower-bound-mu}
There exists a nonnegative function $g\in L^2(\mu,\R^+_0)$ such that
\begin{gather}
  g\,\mu\,\geq\, \frac{c_0}{2}|\nablatx u|. \label{eq:lower-bound-mu}
\end{gather}
In particular, $|\nablatx u|$ is absolutely continuous with respect to
$\mu$,
\begin{gather}
  |\nablatx u|\,<<\, \mu. \label{eq:ac-nabla-u}
\end{gather}
\end{proposition}

\begin{proof}
Let
\begin{gather}
  G(r)\,=\, \int_0^r \sqrt{2 W(s)}\,ds. \label{eq:def-G}
\end{gather}
On the set $\{|\nabla
u_\eps| \neq 0\}$ we have
\begin{align}
  |\nabla G(u_\eps)|= \,&\frac{|\nabla G(u_\eps)|}{|\nablatx G(u_\eps)|}{|\nablatx
  G(u_\eps)|}\notag\\
  =\,&\frac{|\nabla G(u_\eps)|}{\sqrt{\partial_t G(u_\eps)^2 +|\nabla
  G(u_\eps)|^2}}{|\nablatx   G(u_\eps)|}\notag\\
  =\,& \frac{1}{\sqrt{1+\vert v_\eps^2\vert}}{|\nablatx
  G(u_\eps)|}. \label{eq:pre}
\end{align}
Letting $\tilde{\mu}_\eps$ as in \eqref{eq:def-tilde-mu} we get from
\eqref{eq:bound-v-eps}, \eqref{ass:bound-mu}, and Theorem \ref{the:hutch}
the existence of a function $g \in L^2(\mu)$ such that (up to a subsequence) 
\begin{gather}\label{eq:conv-sqrt-mu}
\lim_{\eps\to 0}(\tilde\mu_\eps,\sqrt{1+\vert v_\eps\vert^2})\,=\,(\mu,g)
\end{gather}
as measure-function pairs on $\Omega_T$ with values in $\R$.

Let $\eta\in C^0_c(\Omega_T)$. Then 
\begin{align}
 & \Big\vert \int_{\Omega_T}\eta\sqrt{1+\vert
 v_\eps\vert^2}\,\vert\nabla G(u_\eps)\vert\,dxdt- 
 \int\eta \sqrt{1+\vert v_\eps\vert^2}\,d\tilde\mu_\eps\Big\vert \notag\\ 
 =&\Big\vert \int_{\Omega_T}\eta\sqrt{1+\vert
 v_\eps\vert^2}\Big(\sqrt{\frac{2W(u_\eps)}{\eps}} 
 -\sqrt{\eps}\vert\nabla u_\eps\vert\Big)\sqrt{\eps}\vert\nabla
 u_\eps\vert\,dxdt\Big\vert \notag\\ 
 \leq & \Big(\int_{\Omega_T}\eta^2(1+\vert v_\eps\vert^2)\eps\vert\nabla 
 u_\eps\vert^2\,dxdt\Big)^{1/2}\Big\|\sqrt{\frac{2W(u_\eps)}{\eps}}  
 -\sqrt{\eps}\vert\nabla u_\eps\vert\Big\|_{L^2(\Omega_T)} \notag\\ 
 \leq&\|\eta\|_{L^\infty}(2T\Lenergy+\Lsquares)^{1/2}
 (2\vert\xi_\eps\vert(\Omega_T))^{1/2}.   
 \label{eq:drake}
\end{align}
Thanks to \eqref{eq:conv-sqrt-mu}, \eqref{eq:drake} and
\eqref{eq:zero-xi} we conclude that 
\begin{gather}\label{eq:conv-sqrt-mu-2}
 \lim_{\eps\,\to 0}(\vert\nabla G(u_\eps)\vert\,\LL^{n+1},\sqrt{1+\vert v_\eps\vert^2})
 =(\mu,g)
\end{gather}
as measure-function pairs on $\Omega_T$ with values in $\R$.

Again by \eqref{ass:bound-squares} we have
\begin{align*}
  & \int_{\{0=\vert\nabla u_\eps\vert<W(u_\eps)\}} |\nablatx
  G(u_\eps)|\, dxdt
  \\
  =\, & \int_{\{0=\vert\nabla
  u_\eps\vert<W(u_\eps)\}}\vert\partial_t
  u_\eps\vert\sqrt{2W(u_\eps)}\,dxdt\\ 
  \leq\, & \sqrt 2\Big(\int_{\Omega_T}\eps(\partial_t
  u_\eps)^2\,dxdt\Big)^{1/2}\Big(\int_{\{0=\vert\nabla
  u_\eps\vert<W(u_\eps)\}}\frac{W(u_\eps)}{\eps}\,dxdt\Big)^{1/2}\\ 
  \leq & \sqrt{2\Lsquares}(\vert\xi_\eps\vert(\Omega_T))^{1/2},
\end{align*}
which vanishes by \eqref{eq:zero-xi} as $\eps\to 0$.
This implies together with \eqref{eq:pre} and
\eqref{eq:conv-sqrt-mu-2} that
\begin{equation*}
\begin{split}
  \int\eta \,g\,d\mu=&\, \lim_{\eps\to 0}\int\eta\sqrt{1+\vert
  v_\eps\vert^2}\vert\nabla G(u_\eps)\vert\, dxdt \\
  =&\, \lim_{\eps\to 0} \int_{\Omega_T} \eta |\nablatx G(u_\eps)|\,dxdt
  \geq \frac{c_0}{2} \int_{\Omega_T}\eta\, d\vert\nablatx u\vert,
\end{split}  
\end{equation*}
where in the last line we used that
\begin{gather*}
  \frac{c_0}{2} \int_{\Omega_T} \eta\,d|\nablatx u|\,=\, \int_{\Omega_T} \eta
  \,d|\nablatx G(u)|\,\leq\, \liminf_{\eps\to 0}
  \int_{\Omega_T} \eta |\nablatx G(u_\eps)|\,dxdt.
\end{gather*}
Considering now a set $B\subset
\partial^*\{u=1\}$ with $\mu(B)\,=\, 0$ we conclude that
\begin{gather*}
  |\nablatx u|(B)\,\leq\,\frac{2}{c_0}\int_B g\,d\mu\,=\,0,
\end{gather*}
since $g\in L^2(\mu)$.
\end{proof}

\begin{proposition}\label{prop:rect}
In $\Ha^n$-almost-all points in $\partial^*\{u=1\}$ the
tangential-plane of $\mu$ exists.
\end{proposition}
\begin{proof}
From the Radon-Nikodym Theorem we obtain that the derivative
\begin{gather}
  f(z)\,:=\,D_{|\nablatx u|}\mu(z) \,:=\, \lim_{r\searrow 0}
  \frac{\mu(B_r^{n+1}(z))}{|\nablatx u|(B_r^{n+1}(z))} \label{eq:der-RN}
\end{gather}
exists for $|\nablatx u|$-almost-all $z\in \Omega_T$ and that $f\in
L^1(|\nablatx u|)$. By \eqref{eq:lem:ac} we deduce that 
\begin{gather}
  \mu\lfloor \partial^*\{u=1\} \,=\, f|\nablatx u|. \label{eq:rn}
\end{gather}
Similarly we obtain that 
\begin{gather*}
  \frac{1}{f(z)}\,=\, D_\mu |\nablatx u|(z)
\end{gather*}
is finite for $\mu$-almost all $z\in\partial^*\{u=1\}$. By
\eqref{eq:ac-nabla-u} this implies that
\begin{gather}
  f\,>\,0  \quad |\nablatx u|\text{-almost everywhere in
  }\Omega_T. \label{eq:dens-mu>0} 
\end{gather}
Since $|\nablatx u|$ is rectifiable and $f$ measurable with respect to
$|\nablatx u|$ we obtain from \eqref{eq:rn}, \eqref{eq:dens-mu>0} and
\cite[Remark 11.5]{Si} 
that 
\begin{gather}
  \mu\lfloor \partial^*\{u=1\}\text{ is rectifiable.} \label{eq:rect-rest}
\end{gather}
Moreover $\Ha^n$-almost-all $z\in\partial^*\{u=1\}$ satisfy that
\begin{align}
  &\lim_{r\searrow 0}\frac{\mu(B^{n+1}_r(z)\setminus
  \partial^*\{u=1\})}{\mu(B^{n+1}_r(z))} \,=\, 0, \label{eq:ass-z2}\\
  &\limsup_{r\searrow 0}\frac{\mu(B^{n+1}_r(z))}{\omega_n
  r^n}\,<\, \infty. \label{eq:ass-z3}
\end{align}
In fact,
\eqref{eq:ass-z2} follows from \cite[Theorem 2.9.11]{Fede69} and
\eqref{eq:ac-nabla-u}, and \eqref{eq:ass-z3} from
Proposition \ref{prop:upp-dens-mu} and \eqref{eq:ac-nabla-u}.
Let now $z_0\in\partial^*\{u=1\}$ satisfy
\eqref{eq:ass-z2}, \eqref{eq:ass-z3}.
For an arbitrary $\eta\in C^0_c(B_1^{n+1}(0))$ we then deduce that
\begin{align*}
  &\limsup_{r\to 0} \Big|\int_{\Omega_T\setminus \partial^*\{u=1\}}
  \eta\big(r^{-1}(z-z_0)\big) r^{-n}\,d\mu(z)\Big|
  \\
  \leq\,& \|\eta\|_{C^0_c(B_1^{n+1}(0))} 
  \limsup_{r\to 0} \frac{\mu\big(B_r^{n+1}(z_0)\setminus
  \partial^*\{u=1\}\big)}{\mu\big(B_r^{n+1}(z_0)\big)}
  \limsup_{r\to 0} \frac{\mu\big(B_r^{n+1}(z_0)\big)}{r^n}
  \,=\, 0
\end{align*}
by \eqref{eq:ass-z2}, \eqref{eq:ass-z3}. Therefore
\begin{gather*}
  \lim_{r\to 0} \int_{\Omega_T}
  \eta\big(r^{-1}(z-z_0)\big) r^{-n}\,d\mu(z)\,=\, \lim_{r\to 0}\int_{\partial^*\{u=1\}}
  \eta\big(r^{-1}(z-z_0)\big) r^{-n}\,d\mu(z)
\end{gather*}
if the latter limit exists. By \eqref{eq:rect-rest} we therefore
conclude that in $\Ha^n$-almost-all points of $\partial^*\{u=1\}$ the
tangent-plane of $\mu$ exists and coincides with the tangent plane
of $\mu\lfloor \partial^*\{u=1\}$.
\end{proof}

\begin{proof}[Proof of Proposition \ref{prop:velo}]
Since $u\in BV(\Omega_T)$ and $u(t,\cdot)\in BV(\Omega)$ for almost all
$t\in (0,T)$ we obtain that $\partial_t u, \nabla u$ are Radon measures
on $\Omega_T$ and that $\nabla u(t,\cdot)$ is a
Radon measure on $\Omega$ for almost all $t\in (0,T)$.
Moreover we observe  that $v\in L^1(|\nabla u|)$ since
\begin{gather*}
  \int_{\Omega_T} |v|\,d|\nabla u| \,\leq\, \int_{\Omega_T}
  |v|\,d|\nablatx u| \,\leq\, \frac{2}{c_0}\int_{\Omega_T} g|v|\,d\mu
  \,\leq\, \frac{2}{c_0}\|g\|_{L^2(\mu)}\|v\|_{L^2(\mu)}\,<\, \infty
\end{gather*}
by Theorem \ref{the:lsc-v} and Proposition \ref{prop:lower-bound-mu}.
From \eqref{eq:gen_v_perp} and Proposition \ref{prop:rect} we deduce
that for any $\eta\in C^1_c(\Omega_T)$ 
\begin{gather*}
  -\int_{\Omega_T} \eta\,d\partial_t u \,=\, \int_{\Omega_T} \eta v \,d\nabla
  u\,=\, \int_{\Omega_T} \eta v \cdot\frac{\nabla u}{|\nabla
  u|}\,d|\nabla u| \,=\, \int_0^T\int_\Omega \eta V\, d|\nabla
  u(t,\cdot)|\,dt, 
\end{gather*}
which proves \eqref{eq:def-v-diff}.
\end{proof}
%

\section{Conclusions}
\label{sec:discussion}

Theorem \ref{the:main} suggests to define a generalized action functional
$\A$ in the class of $L^2$-flows by
\begin{gather}
  \A(\mu)\,:=\, \inf_{v} \int_{\Omega_T} |v-H|^2\,d\mu + 4\A_{nuc}(\mu),
  \label{eq:def-S-conl} 
\end{gather}
where the infimum is taken over all  generalized velocities $v$
for the evolution $(\mu^t)_{t\in (0,T)}$. In the class of
$n$-rectifiable $L^2$-flows we have 
\begin{gather}
  \A(\mu)\,=\, \int_{\Omega_T} |v-H|^2\,d\mu + 4\A_{nuc}(\mu),
  \label{eq:def-S-conl-alt} 
\end{gather}
where $v$ is the unique normal velocity of $(\mu^t)_{t\in (0,T)}$ (see
Proposition \ref{prop:uni_gen_vel_perp}).

In the present section we compare the functional $\A$ 
with the functional $\AK$ defined in 
\cite{KORV} (see \eqref{def:red-action}) and discuss the implications of
Theorem 
\ref{the:main} on a full Gamma convergence 
result for the action functional. 
For the ease of the exposition we focus in this section on the switching
scenario.
\begin{assumption}
\label{ass:switch}
Let a sequence $(u_\eps)_{\eps>0}$ of smooth functions
$u_\eps:\Omega_T\to\R$ be given with uniformly bounded action
\eqref{ass:bound}, zero Neumann boundary data \eqref{eq:neumann}, and
assume for the initial- and final states that for all $\eps>0$
\begin{gather}
  u_\eps(0,\cdot)\,=\,-1,\,\quad u_\eps(T,\cdot)\,=\, 1 \qquad\text{ in
  }\Omega.
  \label{eq:bdry-cond-conl}
\end{gather}
\end{assumption}

Following \cite{KORV} we define the reduced action functional on the 
set $\mathcal{M}\subset BV(\Omega_T,\{-1,1\})\cap
L^\infty(0,T,BV(\Omega))$ such that 
\begin{itemize}
\item for every $\psi\in C^0_c(\Omega)$ the function 
\begin{gather*}
  t\mapsto\,\int_\Omega u(t,\cdot)\psi\,dx
\end{gather*}
is absolutely continuous on $[0,T]$;
\item $(\partial^*\{u(t,\cdot)=1\})_{t\in(0,T)}$ 
is up to countably many times given as a smooth evolution of
hypersurfaces. 
\end{itemize}

By Assumption \ref{ass:switch} the functional $\AK_{nuc}$
can be rewritten as
\begin{align}
  \AK(u)\,:=\, &c_0\int_0^T\int_{\Sigma_t} \big|v(t,x) -
  H(t,x)\big|^2 \,d\Ha^{n-1}(x)dt\, + 4\AK_{nuc}(u),
  \label{def:red-action-alt} \\
  \AK_{nuc}(u)\,:=\, &\sum_{t_0\in S}
  \sup_{\psi}\Big(\lim_{t\downarrow t_0} \frac{c_0}{2}|\nabla
  u(t,\cdot)|(\psi)-\lim_{t\uparrow 
  t_0}\frac{c_0}{2}|\nabla u(t,\cdot)|(\psi)\Big) \notag\\
  &+ \sup_\psi \lim_{t\downarrow 0}\frac{c_0}{2}|\nabla
  u(t,\cdot)|(\psi)
  \label{eq:def-S-nuc-alt} 
\end{align}
where the $\sup$ is taken over all $\psi\in C^1(\overline{\Omega})$ with
$0\leq\psi\leq 1$.

In \cite[Proposition 2.2]{KORV} a (formal) proof of the \emph{limsup-
estimate} was given for a subclass of `nice' functions in
$\mathcal{M}$. Following the 
ideas of that proof, using the one-dimensional construction
\cite[Proposition 3.1]{KORV}, and a density argument we expect that the
limsup-estimate can be extended to the whole set $\mathcal{M}$.
We do not give a rigorous proof here but rather assume the limsup-estimate in
the following.
\begin{assumption}\label{ass:limsup}
For all $u\in\mathcal{M}$ there exists a
sequence $(u_\eps)_{\eps>0}$ that satisfies Assumption \ref{ass:switch}
such that   
\begin{gather}
  u\,=\, \lim_{\eps\to 0}u_\eps,\qquad
  \AK(u)\,\geq\, \limsup_{\eps\to 0}\A_\eps(u_\eps). \label{eq:limsup-AK}
\end{gather}
\end{assumption}

The natural candidate for the Gamma-limit of $\A_\eps$ with respect to
$L^1(\Omega_T)$ is the $L^1(\Omega_T)$-lower semicontinuous envelope of
$\AK$, 
\begin{gather}
  \overline{\A}(u)\,:=\, \inf\Big\{ \liminf_{k\to\infty} \AK(u_k)\,:\,
  (u_k)_{k\in\N}\subset \mathcal{M},\, u_k\to u \text{ in }
  L^1(\Omega_T)\Big\}. \label{eq:def-relaxed} 
\end{gather}

\subsection{Comparison of $\A$ and $\AK$}
If we associate with a function $u\in\mathcal{M}$ the measure
$|\nabla u|$ on $\Omega_T$ we can compare  $\AK(u)$
and $\A(\frac{c_0}{2}|\nabla u|)$.
\begin{proposition}
\label{prop:comp-actions}
Let $u\in\mathcal{M}$ and let $\mu=\LL^1\otimes\mu^t$ be an $L^2$-flow
of measures. Assume that for almost
all $t\in (0,T)$ 
\begin{gather}
  \mu^t\,\geq\, \frac{c_0}{2}|\nabla u(t,\cdot)| \label{eq:ass-comp}
\end{gather}
and that the nucleation cost $\AK_{nuc}(u)$ is not larger than the
nucleation cost $\A_{nuc}(\mu)$.
Then
\begin{gather}
  \AK(u)\,\leq\, \A(\mu) \label{eq:comp-actions-1}
\end{gather}
holds. For $\mu=\frac{c_0}{2}|\nabla u|$ we obtain that
\begin{gather}
  \AK(u)\,=\, \A(\frac{c_0}{2}|\nabla u|). \label{eq:coincides}
\end{gather}
\end{proposition}
\begin{proof}
The locality of the mean curvature \cite{Scha07} shows that the weak
mean curvature of $\mu^t$ and the (classical) mean curvature coincide on
$\partial \{u(t,\cdot)=1\}$. By Proposition \ref{prop:velo} 
any generalized velocity $v$ and the (classical) normal velocity
$V$ are equal on the phase boundary. This shows that the integral part
of $\AK(u)$ is not larger than the integral part of $\A(\mu)$, with
equality if $\mu^t=\frac{c_0}{2}|\nabla  u(t,\cdot)|$ for almost all
$t\in (0,T)$. This proves \eqref{eq:comp-actions-1}. For the measure
$\frac{c_0}{2}|\nabla u|$ we observe that the nucleation cost
$\A_{nuc}(\frac{c_0}{2}\mu)$ equals the
nucleation cost $\AK_{nuc}(u)$ and we obtain \eqref{eq:coincides}.
\end{proof}

\begin{figure}
\begin{minipage}{\textwidth}
  \begin{minipage}[b]{.4\linewidth} 
    \includegraphics*[width=6cm]{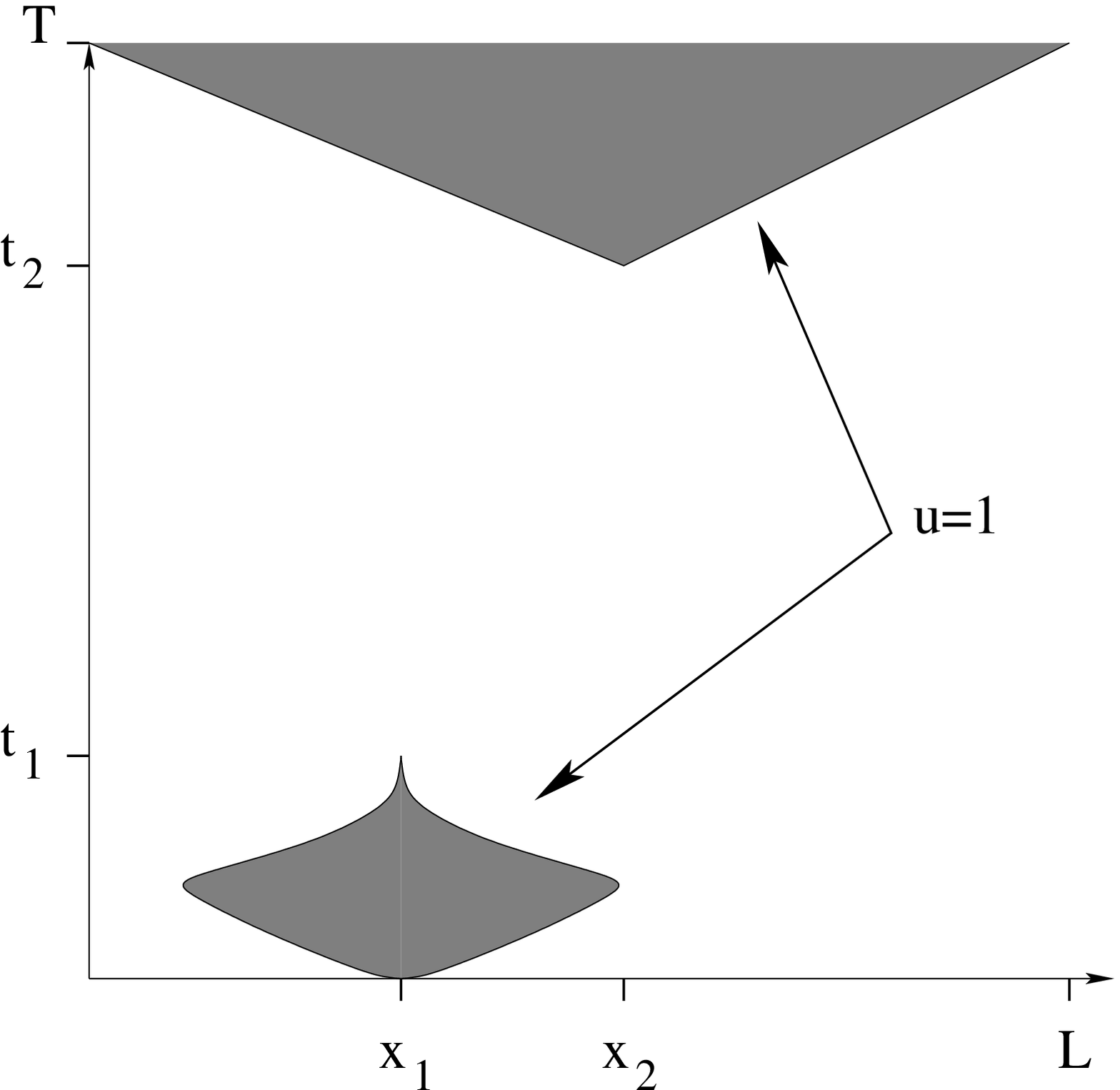}
    \caption{The phases $\{u=1\}$}
    \label{figure:phaseonly}
  \end{minipage}
  \hspace{.1\linewidth}
  \begin{minipage}[b]{.4\linewidth} 
    \includegraphics*[width=6cm]{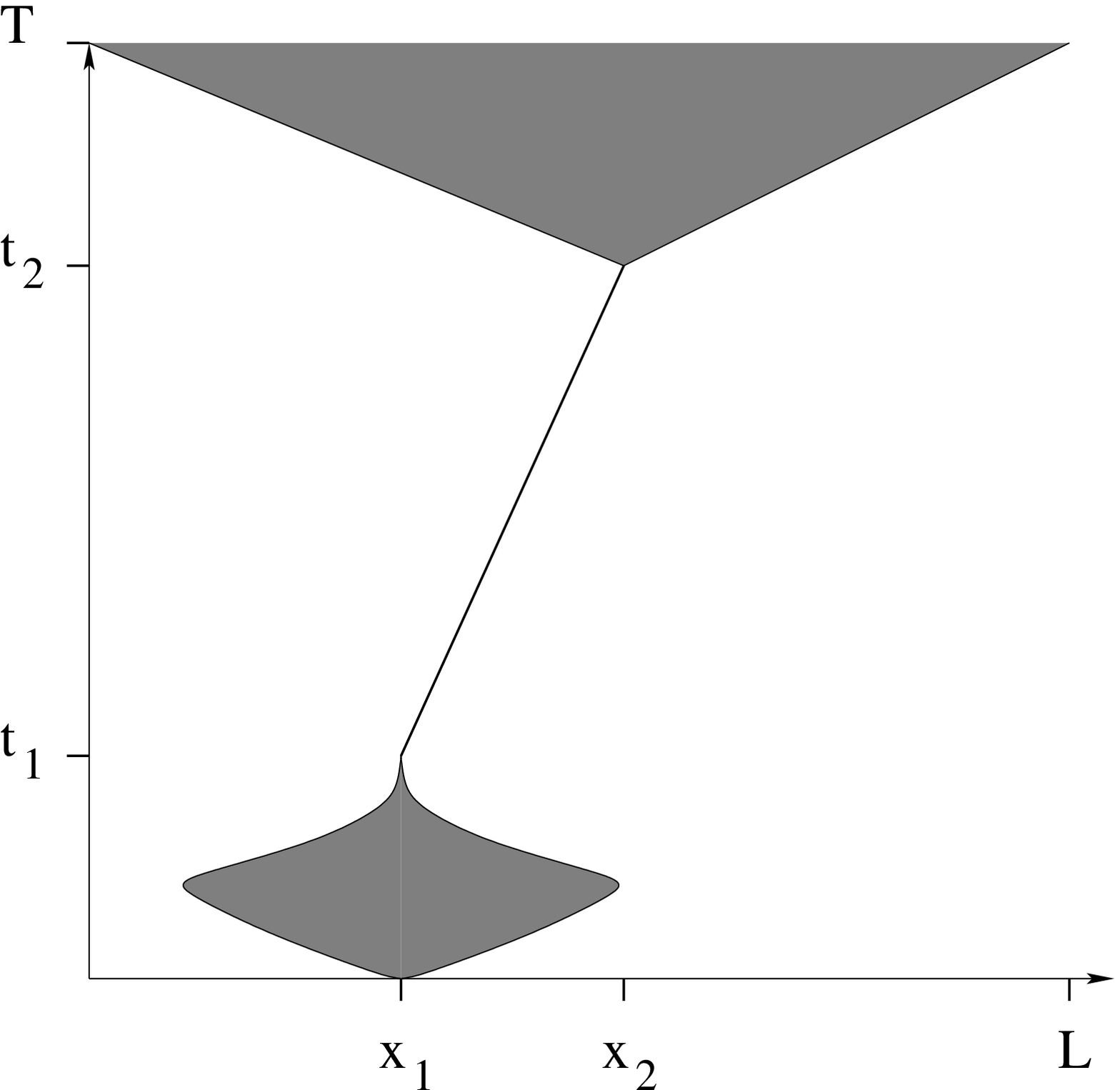}
    \caption{The measure $\mu$}
    \label{figure:muonly}
  \end{minipage}
\end{minipage}
\end{figure}

If higher multiplicities occur for the measure $\mu$, the nucleation costs
of $\mu$ and $u$ may differ and the value of 
$\AK(u)$ might be larger than $\A(\mu)$ as the following example
shows. Let $\Omega=(0,L)$, let $\{u=1\}$
be  the shaded regions in Figure \ref{figure:phaseonly}, and let $\mu$ be the
measure supported on the phase boundary and with double density on a
hidden boundary connecting the upper and lower part of the phase
$\{u=1\}$, see Figure \ref{figure:muonly}. At time $t_2$ a new phase is
nucleated but this time is not singular with respect to the
evolution $(\mu^t)_{t\in (0,T)}$. On the other hand, no
propagation cost occurs for the
evolution $(u(t,\cdot))_{t\in (t_1,t_0)}$ whereas there is a propagation
cost for 
$(\mu^t)_{t\in (t_1,t_2)}$. The difference in
action is given by
\begin{gather*}
  \AK(u)-\A(\mu) \,=\, 8c_0 - {2c_0}\frac{(x_2-x_1)^2}{t_2-t_1}.
\end{gather*}
where $x_1$ is the annihilation point at time $t_1$ and $x_2$ the
nucleation point at time $t_2$, see Figure \ref{figure:phaseonly}.
This shows that as soon as $(x_2-x_1)<4\sqrt{t_2-t_1}$ we have
\begin{gather*}
  \A(\mu)\,<\, \AK(u).
\end{gather*}
\begin{figure}[h!]
\begin{minipage}{\textwidth}
  \begin{minipage}[b]{.4\linewidth} 
    \includegraphics*[width=6cm]{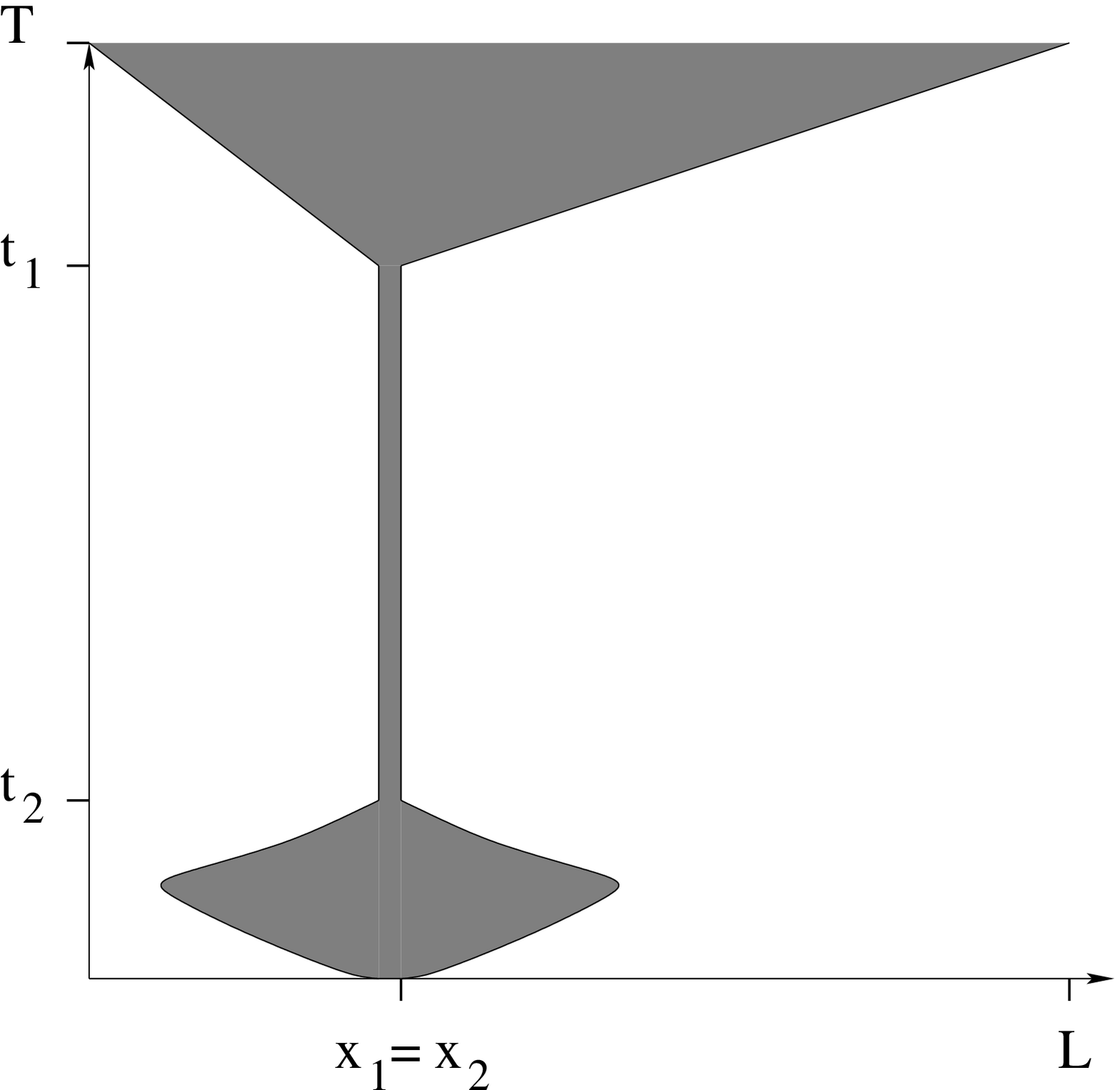}
    \caption{\mbox{Phases $\{u_k=1\}$}}
    \label{figure:phasek}
  \end{minipage}
  \hspace{.1\linewidth}
  \begin{minipage}[b]{.4\linewidth} 
    \includegraphics*[width=6cm]{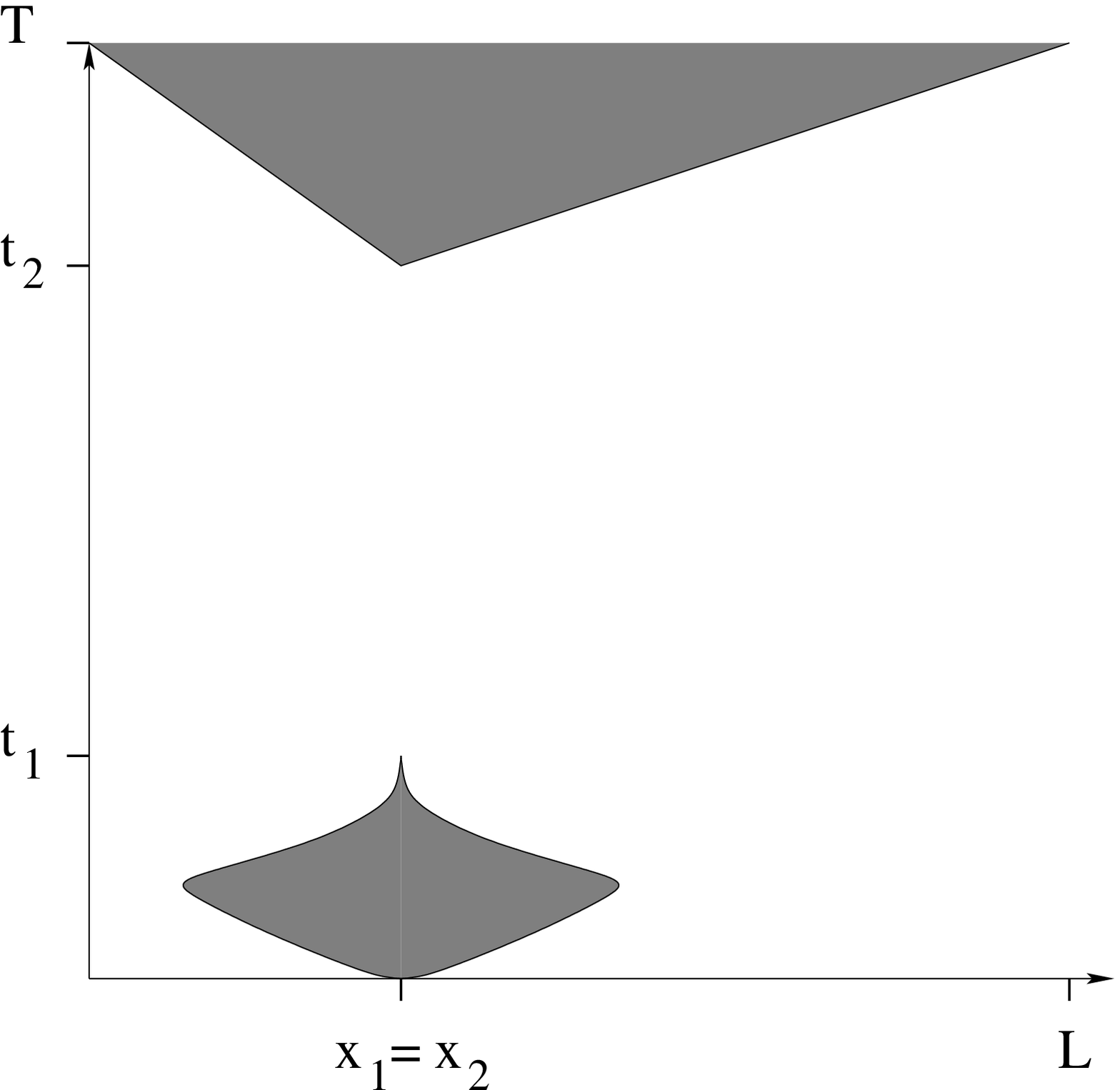}
    \caption{The limit}
    \label{figure:phaselim}
  \end{minipage}
\end{minipage}
\end{figure}

The same example with $x_2=x_1$ shows that $\AK$ is not
lower-semicontinuous and that a relaxation is necessary in 
order to obtain the Gamma-limit of $\A_\eps$. In fact consider a
sequence $(u_k)_{k\in\N}$ with phases $\{u_k=1\}$ given by the shaded
region in Figure \ref{figure:phasek}. Assume that the neck connecting the upper
and lower part of the shaded region disappears with $k\to\infty$ and
that $u_k$ converges to the phase indicator function $u$ with phase
$\{u=1\}$ indicated by the shaded regions in Figure \ref{figure:phaselim}. Then a
nucleation cost at time $t_2$ appears for $u$. For the approximations
$u_k$ however 
there is no nucleation cost for $t>0$ and the approximation can be made
such that the propagation cost in $(t_1,t_2)$ is arbitrarily small,
which shows that 
\begin{gather*}
  \AK(u)\,>\, \liminf_{k\to\infty} \AK(u_k).
\end{gather*}
The situation in higher space dimensions is even more involved than in
the one-dimensional examples discussed above. For instance one could
create a circle with double density (no new phase is created) at a time
$t_1$ and let this double-density circle grow until a time $t_2>t_1$
where the double-density circle splits and two circles evolve in
different directions, 
one of them shrinking and the other one growing. In this way a new phase
is created at time $t_2$. In this example $\A$ counts the creation
of a double-density circle at time $t_1$ and the cost of propagating the
double-density circle between the times $t_1,t_2$. In contrast $\AK$
counts the nucleation cost of the new phase at time $t_2$, which is
larger as the nucleation cost $\A_{nuc}$ at times $t_1$, but no
propagation cost between the times $t_1,t_2$. 

The analysis in \cite{KORV} suggests that minimizers of the action
functional exhibit nucleation and annihilation of phases only at the
initial- and final time. This class is therefore particularly
interesting.
\begin{theorem}\label{the:concl}
Let $(u_\eps)_{\eps>0}$ satisfy Assumption \ref{ass:switch} and suppose
that Assumption \ref{ass:limsup} holds.
Suppose that $u_\eps\to u$ in $L^1(\Omega_T)$, $u\in\mathcal{M}$, and
that $u$ exhibits nucleation and annihilation of phases only at the final
and initial time. Then 
\begin{gather}
  \overline{\A}(u)\,=\,\AK(u)\,\leq\,\liminf_{\eps\to
  0}\A_\eps(u_{\eps}) \label{eq:comp-concl}
\end{gather}
holds. In particular, $\A_\eps$ Gamma-converges to $\AK$ for those
evolutions in $\mathcal{M}$ that have nucleations only at the initial
time.  
\end{theorem}
\begin{proof}
From the definition of the functional $\overline{\A}$ we deduce that
\begin{gather}
  \overline{\A}(u)\,\leq\, \AK(u) \label{eq:first-comp}
\end{gather}
and that
there exists a sequence $(u_k)_{k\in\N}\subset\mathcal{M}$ such that
\begin{gather*}
  u\,=\, \lim_{k\to\infty} u_k,\qquad
  \overline{\A}(u)\,=\,\lim_{k\to\infty} \AK(u_k).
\end{gather*}
Assumption \ref{ass:limsup} implies that for all $k\in\N$
there exists a sequence $(u_{\eps,k})_{\eps>0}$ such that
\begin{gather}
   u_k\,=\,\lim_{\eps\to 0}u_{\eps,k},\qquad
  \AK(u_k)\,\geq\, \limsup_{\eps\to 0}\A_\eps(u_{\eps,k}). \label{eq:geq-eps-k}
\end{gather}
Therefore we can choose a diagonal-sequence $(u_{\eps(k),k})_{k\in\N}$ such
that
\begin{gather}
  \overline{\A}(u)\,\geq\, \limsup_{k\to
  \infty}\A_{\eps(k)}(u_{\eps(k),k}). \label{eq:leq-eps-k} 
\end{gather}
By Proposition \ref{prop:MM}, \ref{prop:KRT} there exists a
a subsequence $k\to\infty$ such that
\begin{gather}
  u_{\eps(k),k}\,\to\, u,\quad \mu_{\eps(k),k}\,\to\, \mu,\quad \mu\,\geq\,
  \frac{c_0}{2}|\nabla u|, \label{eq:conv-eps-k}
\end{gather}
where the last inequality follows from
\begin{gather*}
  \frac{c_0}{2}\int_\Omega \eta \,d|\nabla u(t,\cdot)|\,\leq\,
  \liminf_{\eps\to 0}\int_\Omega\eta |\nabla G(u_\eps)|\,dx\,\leq\,
  \liminf_{\eps\to 0}\int_\Omega \eta d\mu_\eps^t\,=\, \int_\Omega \eta
  d\mu^t,
\end{gather*}
with $G$ as in \eqref{eq:def-G}.
By Theorem \ref{the:main} we further deduce that
\begin{gather*}
  \liminf_{k\to \infty}\A_{\eps(k)}(u_{\eps(k),k})\,\geq\, \A(\mu).
\end{gather*}
This implies by \eqref{eq:leq-eps-k} that
\begin{gather}
  \overline{\A}(u)\,\geq\, \A(\mu). \label{eq:comp-2}
\end{gather}
Since $\mu^0=0$ and $\mu^t\geq  \frac{c_0}{2}|\nabla u(t,\cdot)|$ the nucleation
cost of $\mu$ at $t=0$ is not lower than the nucleation cost for
$u$. Since by assumption there are no more nucleation times
we can apply Proposition \ref{prop:comp-actions} and obtain that
$\AK(u)\,\leq\, \A(\mu)$. By \eqref{eq:first-comp}, \eqref{eq:comp-2}
we conclude that $\AK(u)=\overline{\A}(u)= \A(\mu)$.

Applying Proposition \ref{prop:MM} and  Theorem \ref{the:main} to the
sequence $(u_\eps)_{\eps>0}$ we deduce that there exists a subsequence
$\eps\to 0$ such that
\begin{gather}
  \mu_{\eps}\,\to\, \tilde{\mu},\qquad
  \tilde{\mu}\,\geq\,   \frac{c_0}{2}|\nabla u| \label{eq:conv-eps-tilde}
\end{gather}
and such that
\begin{gather*}
  \liminf_{\eps\to 0}\A_\eps(u_{\eps})\,\geq\, \A(\tilde{\mu}).
\end{gather*}
Repeating the arguments above we deduce from Proposition
\ref{prop:comp-actions} that $\AK(u)\leq \A(\tilde{\mu})$ and
\begin{gather*}
  \AK(u)\,\leq\, \liminf_{\eps\to 0} \A_\eps(u_\eps).
\end{gather*}
Combining the upper bound \eqref{eq:limsup-AK} with
\eqref{eq:comp-concl} proves the Gamma convergence of $\A_\eps$ in $u$.
\end{proof}

\subsection{Gamma convergence under an additional assumption}
Using Theorem \ref{the:main} we can prove the Gamma convergence of
$\A_\eps$ 
under an additional assumption on  the structure of the set of those
measures  
that arise as limit of sequences with uniformly bounded action.
\begin{assumption}
\label{ass:liminf}
Consider any sequence $(u_\eps)_{\eps>0}$ with $u_\eps\,\to\, u$ in $L^1(\Omega_T)$
that satisfies Assumption \ref{ass:switch}. Define the
energy measures 
$\mu_\eps$ according to \eqref{eq:def-mu-eps} and let $\mu$ be any Radon
measure such that for a subsequence $\eps\to 0$
\begin{gather}
  \mu=\lim_{\eps\to 0}\mu_\eps. \label{eq:ass-lim-mu}
\end{gather}
Then we assume that there exists a sequence
$(u_k)_{k\in\N}\subset\mathcal{M}$ such that 
\begin{align}
  u\,=\, \lim_{k\to\infty} u_k, \quad
  \A(\mu)\,\geq\, \lim_{k\to\infty} \AK(u_k). \label{eq:conv-S-u-approx}
\end{align}
\end{assumption}
For any $u\in\mathcal{M}$ that exhibits nucleation and annihilation only
at initial and final time the Assumption \ref{ass:liminf} is always
satisfied: The proof of Theorem \ref{the:concl} and our
results in Section \ref{sec:results} show that for any limit $\mu$ as in
\eqref{eq:ass-lim-mu} we can apply Proposition
\ref{prop:comp-actions}. Therefore $\AK(u)\leq 
\A(\mu)$ and the constant sequence $u$ satisfies
\eqref{eq:conv-S-u-approx}. 
However, a characterization of those $u\in\mathcal{M}$ such that Assumption
\ref{ass:liminf} holds is open.
\begin{theorem}
\label{the:Gamma-convergence}
Suppose that the Assumptions \ref{ass:switch}, \ref{ass:limsup}, and
\ref{ass:liminf} hold. Then
\begin{gather}
  \A_\eps\,\to\, \overline{\A}\quad\text{ as }\eps\to 0
  \label{eq:Gamma-convergence} 
\end{gather}
in the sense of Gamma-convergence with respect to $L^1(\Omega_T)$.
\end{theorem}
\begin{proof}
We first prove the \emph{limsup}-estimate for
$\A_\eps,\overline{\A}$. In fact, fix an arbitrary $u\in 
L^1(\Omega_T,\{-1,1\})$ with $\overline{\A}(u)<\infty$. We 
deduce that there exists a sequence $(u_k)_{k\in\N}$ as in
\eqref{eq:def-relaxed} such that 
\begin{align}
  \overline{\A}(u)\,&=\, \lim_{k\to \infty} \AK(u_k). \label{eq:lim-AK}
\end{align}
By \eqref{eq:limsup-AK} for all $k\in\N$ there exists a sequence
$(u_{\eps,k})_{\eps>0}$ such that 
\begin{gather*}
  \lim_{\eps\to 0}u_{\eps,k}\,=\, u_k,\qquad
  \AK(u_k)\,\geq\, \limsup_{\eps\to
  0}\A_\eps(u_{\eps,k}).
\end{gather*}
Choosing a suitable diagonal sequence $u_{\eps(k),k}$ we
deduce that
\begin{gather}
  \overline{\A}(u)\,\geq\, \lim_{k\to \infty} \A_{\eps(k)}(u_{\eps(k),k}),
  \label{eq:limsup} 
\end{gather}
which proves the \emph{limsup}-estimate.

We next prove the \emph{liminf}-estimate. 
Consider an arbitrary sequence $(u_\eps)_{\eps>0}$ that satisfies the
Assumption \ref{ass:switch}. By Theorem \ref{the:main} there exists $u\in
BV(\Omega_T,\{-1,1\})$ and a measure $\mu$ on $\Omega_T$ such that
\begin{gather}
  u_\eps\,\to\, u\quad\text{ in }L^1(\Omega_T),\qquad
  \mu_\eps\,\to\,\mu
\end{gather}
for a subsequence $\eps\to 0$, and such that
\begin{gather}
  \liminf_{\eps\to 0}\A_\eps(u_\eps)\,\geq\, \A(\mu). \label{eq:fast}
\end{gather}
By Assumption \ref{ass:liminf} there exists a sequence
$(u_k)_{k\in\N}\subset\mathcal{M}$ such that \eqref{eq:conv-S-u-approx}
holds. By 
\eqref{eq:fast} and the definition of $\overline{\A}$ this yields that 
\begin{gather}
  \liminf_{\eps\to 0}\A_\eps(u_\eps)\,\geq\, \A(\mu)\,\geq\, \lim_{k\to\infty}
  \AK(u_k)\,\geq\, \overline{\A}(u)
\end{gather}
and proves the \emph{liminf}-estimate.
\end{proof}
\begin{appendix}
\section{Rectifiable measures and weak mean curvature}
\label{app:rect}
We briefly summarize some definitions from Geometric Measure Theory.
We always restrict ourselves to the
hypersurface case, that is `tangential-plane' and `rectifiability' of a
measure in $\R^d$ means `$(d-1)$-dimensional tangential-plane' and
`$(d-1)$-rectifiable'. 
\begin{definition}
\label{def:varifolds}
Let $\mu$ be a Radon-measure in $\R^d$, $d\in\N$.
\begin{enumerate}
\item
We say that $\mu$ has a \emph{(generalized) tangent
plane} in $z\in\R^d$ if there exist a number $\Theta>0$ and a
$(d-1)$-dimensional linear subspace $T\subset\R^d$ such that  
\begin{gather}
 \lim_{r\searrow 0}
 r^{-d+1}\int\eta\left(\frac{y-z}{r}\right)\,d\mu(y)\,=\,\Theta
 \int_{T}\eta\,d\Ha^{d-1}, \qquad 
 \text{ for every } \eta\in C^0_c(\R^d). \label{eq:def-Tan}
\end{gather}
We then set $T_z\mu\,:=\, T$ and call $\Theta$ the multiplicity of
$\mu$ in $z$.
\item
If for $\mu$-almost all $z\in\R^d$ a tangential plane exists then we
call $\mu$ \emph{rectifiable}. If in addition the multiplicity is
integer-valued $\mu$-almost everywhere we say that $\mu$ is
\emph{integer-rectifiable}.
\item
The \emph{first variation} $\delta\mu: C^1_c(\R^d,\R^d)$ of a
rectifiable Radon-measure $\mu$ on $\R^d$ is defined by
\begin{gather*}
  \delta\mu(\eta)\,:=\, \int \dive_{T_z\mu}\eta\,d\mu.
\end{gather*}
If there exists a function ${H}\in L^1_\loc(\mu)$ such that
\begin{gather*}
  \delta\mu(\eta)\,=\, -\int {H}\cdot\eta\,d\mu
\end{gather*}
we call ${H}$ the weak mean-curvature vector of $\mu$.
\end{enumerate}
\end{definition}
%
\section{Measure-function pairs}
\label{app:hutch}
We recall
some basic facts about the notion of \emph{measure function pairs}
introduced by Hutchinson in \cite{Hu}. 
\begin{definition}\label{def:meas-funct-pair}
Let $E\subset\R^d$ be an open subset. Let $\mu$ be a positive Radon-measure
on $E$. Suppose $f:E\,\to\,\R^m$ is well defined $\mu$-almost everywhere, and
$f\in L^1(\mu,\R^m)$. Then we say $(\mu, f)$ is a measure-function pair over $E$ (with 
values in $\R^m$).
\end{definition}
Next we define two notions of convergence for a sequence of measure-function pairs
on $E$ with values in $\R^m$.
\begin{definition}\label{def:conv-meas-funct-pairs}
Suppose $\{(\mu_k,\,f_k)\}_k$ and $(\mu,\, f)$ are measure-function pairs over $E$ with values in
$\R^m$. Suppose 
\begin{equation*}
 \lim_{k\to\infty}\mu_k=\,\mu,
\end{equation*} 
as Radon-measures on $E$. Then we say $(\mu_k,\,f_k)$ converges to $(\mu,\,f)$ in the weak sense 
(in $E$) and write
\begin{equation*}
 (\mu_k,\, f_k)\to\,(\mu,\, f),
\end{equation*}
if $\mu_k \lfloor f_k\to\,\mu\lfloor f$ in the sense of vector-valued
measures, that means
\begin{equation*}
 \lim_{k\to\infty}\int f_k\cdot\eta\,d\mu_k=\int f\cdot \eta\,d\mu,
\end{equation*}
for all $\eta\in C^0_c(E,\R^m)$.
\end{definition}
The following result is a slightly less general version of \cite[Theorem 4.4.2]{Hu},
however this is enough for our aims.
\begin{theorem}\label{the:hutch}
Let $F:\R^m\,\to\,[0,+\infty)$ be a continuous, convex function with super-linear growth
at infinity, that is:
\begin{equation*}
\lim_{\vert y\vert\to\infty}\frac{F(y)}{\vert y\vert}=+\infty.
\end{equation*}
Suppose $\{(\mu_k,\, f_k)\}_k$ are measure-function pairs over $E\subset \R^d$ with values
in $\R^m$. Suppose $\mu$ is Radon-measure on $E$ and $\mu_k\to\mu$ as $k\to\infty$.
Then the following are true:
\begin{enumerate}
\item 
 if
\begin{equation}
 \sup_{k}\int \int F(f_k)\,d\mu_k <+\infty,
\label{eq:ener-bound-Hu} 
\end{equation}
then some subsequence of $\{(\mu_k,\, f_k)\}$ converges in the weak sense to some
measure function $(\mu,\,f)$ for some $f$.
\item
 if \eqref{eq:ener-bound-Hu} holds and $(\mu_k,\, f_k)\to(\mu,\, f)$
 then
\begin{equation}
 \liminf_{k\to\infty}\int F(f_k)\,d\mu_k\geq\int F(f)\,d\mu.
\label{eq:lsc-Hu} 
\end{equation} 
\end{enumerate}
\end{theorem}
\end{appendix}

%
\def\cprime{$'$}


\begin{thebibliography}{10}

\bibitem{Alla72}
William~K. Allard.
\newblock On the first variation of a varifold.
\newblock {\em Ann. of Math. (2)}, 95:417--491, 1972.

\bibitem{ALan03}
H.~Allouba and J.~A. Langa.
\newblock Semimartingale attractors for {A}llen-{C}ahn {SPDE}s driven by
  space-time white noise. {I}. {E}xistence and finite dimensional asymptotic
  behavior.
\newblock {\em Stoch. Dyn.}, 4(2):223--244, 2004.

\bibitem{AFPa00}
Luigi Ambrosio, Nicola Fusco, and Diego Pallara.
\newblock {\em Functions of bounded variation and free discontinuity problems}.
\newblock Oxford Mathematical Monographs. The Clarendon Press Oxford University
  Press, New York, 2000.

\bibitem{BMug07}
Giovanni Bellettini and Luca Mugnai.
\newblock Remarks on the variational nature of the heat equation and of mean
  curvature flow.
\newblock preprint, 2007.

\bibitem{Brak78}
Kenneth~A. Brakke.
\newblock {\em The motion of a surface by its mean curvature}, volume~20 of
  {\em Mathematical Notes}.
\newblock Princeton University Press, Princeton, N.J., 1978.

\bibitem{DG}
Ennio De~Giorgi.
\newblock Some remarks on {$\Gamma$}-convergence and least squares method.
\newblock In {\em Composite media and homogenization theory (Trieste, 1990)},
  volume~5 of {\em Progr. Nonlinear Differential Equations Appl.}, pages
  135--142. Birkh\"auser Boston, Boston, MA, 1991.

\bibitem{dMSc90}
Piero de~Mottoni and Michelle Schatzman.
\newblock Development of interfaces in {${\bf R}\sp N$}.
\newblock {\em Proc. Roy. Soc. Edinburgh Sect. A}, 116(3-4):207--220, 1990.

\bibitem{ERVa04}
Weinan E, Weiqing Ren, and Eric Vanden-Eijnden.
\newblock Minimum action method for the study of rare events.
\newblock {\em Comm. Pure Appl. Math.}, 57(5):637--656, 2004.

\bibitem{ESSo92}
Lawrence~C. Evans, Halil~Mete Soner, and Panagiotis~E. Souganidis.
\newblock Phase transitions and generalized motion by mean curvature.
\newblock {\em Comm. Pure Appl. Math.}, 45(9):1097--1123, 1992.

\bibitem{FJon82}
William~G. Faris and Giovanni Jona-Lasinio.
\newblock Large fluctuations for a nonlinear heat equation with noise.
\newblock {\em J. Phys. A}, 15(10):3025--3055, 1982.

\bibitem{Fede69}
Herbert Federer.
\newblock {\em Geometric measure theory}.
\newblock Die Grundlehren der mathematischen Wissenschaften, Band 153.
  Springer-Verlag New York Inc., New York, 1969.

\bibitem{Feng06}
Jin Feng.
\newblock Large deviation for diffusions and {H}amilton-{J}acobi equation in
  {H}ilbert spaces.
\newblock {\em Ann. Probab.}, 34(1):321--385, 2006.

\bibitem{FHSv04}
Hans~C. Fogedby, John Hertz, and Axel Svane.
\newblock Domain wall propagation and nucleation in a metastable two-level
  system, 2004.

\bibitem{WFre79}
Mark~I. Freidlin and Alexander~D. Wentzell.
\newblock {\em Fluktuatsii v dinamicheskikh sistemakh pod deistviem malykh
  sluchainykh vozmushchenii}.
\newblock ``Nauka'', Moscow, 1979.
\newblock Teoriya Veroyatnostei i Matematicheskaya Statistika. [Probability
  Theory and Mathematical Statistics].

\bibitem{FWen98}
Mark~I. Freidlin and Alexander~D. Wentzell.
\newblock {\em Random perturbations of dynamical systems}, volume 260 of {\em
  Grundlehren der Mathematischen Wissenschaften [Fundamental Principles of
  Mathematical Sciences]}.
\newblock Springer-Verlag, New York, second edition, 1998.
\newblock Translated from the 1979 Russian original by Joseph Sz\"ucs.

\bibitem{Hu}
John~E. Hutchinson.
\newblock Second fundamental form for varifolds and the existence of surfaces
  minimising curvature.
\newblock {\em Indiana Univ. Math. J.}, 35(1):45--71, 1986.

\bibitem{Ilma93}
Tom Ilmanen.
\newblock Convergence of the {A}llen-{C}ahn equation to {B}rakke's motion by
  mean curvature.
\newblock {\em J. Differential Geom.}, 38(2):417--461, 1993.

\bibitem{KORV}
Robert Kohn, Felix Otto, Maria~G. Reznikoff, and Eric Vanden-Eijnden.
\newblock Action minimization and sharp-interface limits for the stochastic
  {A}llen-{C}ahn equation.
\newblock {\em Comm. Pure Appl. Math.}, 60(3):393--438, 2007.

\bibitem{KRT}
Robert~V. Kohn, Maria~G. Reznikoff, and Yoshihiro Tonegawa.
\newblock Sharp-interface limit of the {A}llen-{C}ahn action functional in one
  space dimension.
\newblock {\em Calc. Var. Partial Differential Equations}, 25(4):503--534,
  2006.

\bibitem{KRVa05}
Robert~V. Kohn, Maria~G. Reznikoff, and Eric Vanden-Eijnden.
\newblock Magnetic elements at finite temperature and large deviation theory.
\newblock {\em J. Nonlinear Sci.}, 15(4):223--253, 2005.

\bibitem{KKLa07}
Georgios T.~Kossioris Markos A.~Katsoulakis and Omar Lakkis.
\newblock Noise regularization and computations for the 1-dimensional
  stochastic {A}llen-{C}ahn problem.
\newblock {\em Interfaces Free Bound.}, 9(1):1--30, 2007.

\bibitem{Mo}
Luciano Modica.
\newblock The gradient theory of phase transitions and the minimal interface
  criterion.
\newblock {\em Arch. Rational Mech. Anal.}, 98(2):123--142, 1987.

\bibitem{MM}
Luciano Modica and Stefano Mortola.
\newblock Un esempio di {$\Gamma$}-convergenza.
\newblock {\em Boll. Un. Mat. Ital. B (5)}, 14(1):285--299, 1977.

\bibitem{Mose01}
Roger Moser.
\newblock A generalization of {R}ellich's theorem and regularity of varifolds
  minimizing curvature.

\bibitem{PS}
Pavel~I. Plotnikov and Victor~N. Starovo{\u\i}tov.
\newblock The {S}tefan problem with surface tension as a limit of the phase
  field model.
\newblock {\em Differentsial\cprime nye Uravneniya}, 29(3):461--471, 550, 1993.

\bibitem{RTon07}
Maria~G. Reznikoff and Yoshihiro Tonegawa.
\newblock Higher multiplicity in the one-dimensional {A}llen-{C}ahn action
  functional.
\newblock preprint, 2007.

\bibitem{Ro}
Matthias R{\"o}ger.
\newblock Existence of weak solutions for the {M}ullins-{S}ekerka flow.
\newblock {\em SIAM J. Math. Anal.}, 37(1):291--301, 2005.

\bibitem{RS}
Matthias R\"oger and Reiner Sch\"atzle.
\newblock On a modified conjecture of {D}e {G}iorgi.
\newblock {\em Mathematische Zeitschrift}, 254(4):675--714, 2006.

\bibitem{Scha07}
Reiner Sch\"atzle.
\newblock Lower semicontinuity of the {W}illmore functional for currents.
\newblock preprint, 2007.

\bibitem{Shar00}
Tony Shardlow.
\newblock Stochastic perturbations of the {A}llen-{C}ahn equation.
\newblock {\em Electron. J. Differential Equations}, pages No. 47, 19 pp.
  (electronic), 2000.

\bibitem{Si}
Leon Simon.
\newblock {\em Lectures on geometric measure theory}, volume~3 of {\em
  Proceedings of the Centre for Mathematical Analysis, Australian National
  University}.
\newblock Australian National University Centre for Mathematical Analysis,
  Canberra, 1983.

\end{thebibliography}
\end{document}